\documentclass[leqno,a4paper,12pt]{scrartcl}
\usepackage{amsmath,upref,amssymb,amsthm,amsxtra,exscale}
\usepackage{cite}
\usepackage{enumitem}
\usepackage[T1]{fontenc}
\usepackage{lmodern}
\usepackage[utf8]{inputenc}
\usepackage{csquotes}
\usepackage{mathtools}
\usepackage{ifthen}
\usepackage[reftex]{theoremref}
\numberwithin{equation}{section}

\theoremstyle{plain}
\newtheorem{theorem}{Theorem}[section]
\newtheorem{lemma}[theorem]{Lemma}
\newtheorem{proposition}[theorem]{Proposition}
\newtheorem{corollary}[theorem]{Corollary}

\theoremstyle{definition}

\newtheorem{definition}[theorem]{Definition}

\newcommand{\prob}[1]{\textup{(}$P_{#1}$\textup{)}}
\newcommand{\fprob}[1]{\textup{(}$F_{#1}$\textup{)}}
\newcommand{\olambda}{\ensuremath{\bar{\lambda}}}

\newcommand{\eps}{\varepsilon}

\newcommand{\dC}{\mathbb{C}}

\newcommand{\dN}{\mathbb{N}}

\newcommand{\dR}{\mathbb{R}}

\newcommand{\dZ}{\mathbb{Z}}

\newcommand{\cH}{\mathcal{H}}

\newcommand{\cL}{\mathcal{L}}

\newcommand{\cN}{\mathcal{N}}

\newcommand{\cR}{\mathcal{R}}

\newcommand{\cT}{\mathcal{T}}

\newcommand{\rmD}{\mathrm{D}}

\newcommand{\rmd}{\mathrm{d}}
\newcommand{\rme}{\mathrm{e}}
\newcommand{\rmf}{\mathrm{f}}

\newcommand{\olu}{{\bar{u}}}

\newcommand{\wtu}{\tilde{u}}

\DeclareMathOperator{\opspan}{span}
\DeclareMathOperator{\codim}{codim}

\DeclareMathOperator{\Real}{Re}
\DeclareMathOperator{\Imag}{Im}

\DeclareMathOperator*{\wlim}{w-lim}

\newcommand{\rmloc}{\mathrm{loc}}
\newcommand{\rmess}{\mathrm{ess}}

\newcommand{\dint}{\,\rmd}
\newcommand{\ssm}{\backslash}
\newcommand{\weakto}{\rightharpoonup}

\newcommand{\lr}[3]{#1#3#2}
\newcommand{\xlr}[3]{\left#1#3\right#2}
\newcommand{\biglr}[3]{\bigl#1#3\bigr#2}
\newcommand{\Biglr}[3]{\Bigl#1#3\Bigr#2}
\newcommand{\bigglr}[3]{\biggl#1#3\biggr#2}

\newcommand{\abs}[1]{\lr\lvert\rvert{#1}}
\newcommand{\xabs}[1]{\xlr\lvert\rvert{#1}}
\newcommand{\bigabs}[1]{\biglr\lvert\rvert{#1}}
\newcommand{\Bigabs}[1]{\Biglr\lvert\rvert{#1}}

\newcommand{\norm}[1]{\lr\lVert\rVert{#1}}

\newcommand{\biggnorm}[1]{\bigglr\lVert\rVert{#1}}

\newcommand{\scp}[1]{\lr\langle\rangle{#1}}

\newcommand{\bigscp}[1]{\biglr\langle\rangle{#1}}

\newcommand{\coloneqq}{:=}

\newcommand{\loc}{\mathrm{loc}}

\begin{document}

\title{Unstable normalized standing waves for the space periodic
  NLS}

\author{Nils Ackermann\thanks{Supported by CONACYT grant 237661,
    UNAM-DGAPA-PAPIIT grant IN100718 and the program
    UNAM-DGAPA-PASPA (Mexico)}\and Tobias
  Weth} \date{}

\maketitle
\begin{abstract}
  For the stationary nonlinear Schrödinger equation
  $-\Delta u+ V(x)u- f(u) = \lambda u$ with periodic potential
  $V$ we study the existence and stability properties of
  multibump solutions with prescribed $L^2$-norm. To this end we
  introduce a new nondegeneracy condition and develop new
  superposition techniques which allow to match the
  $L^2$-constraint.  In this way we obtain the existence of
  infinitely many geometrically distinct solutions to the
  stationary problem.  We then calculate the Morse index of these
  solutions with respect to the restriction of the underlying
  energy functional to the associated $L^2$-sphere, and we show
  their orbital instability with respect to the Schrödinger
  flow. Our results apply in both, the mass-subcritical and the
  mass-supercritical regime.


\end{abstract}

\section{Introduction}
\label{sec:introduction}

Suppose that $N\in\dN$ and consider the stationary nonlinear
Schrödinger equation with prescribed $L^2$-norm
\begin{equation*}
  \tag*{\prob{\alpha}}
  -\Delta u+V(x)u- f(u)=\lambda u, \qquad u\in H^1(\dR^N), \qquad  \abs{u}_2^2=\alpha
\end{equation*}
which we will call the \emph{constrained} equation.  Here
$\abs{\cdot}_2$ denotes the standard $L^2$-norm,
$V\in L^\infty(\dR^N)$ is periodic in all coordinates, $f$ is a
superlinear nonlinearity of class $C^1$ with Sobolev-subcritical
growth, $\alpha>0$ is given, $u$ is the unknown weak solution and
$\lambda\in\dR$ is an unknown parameter.  Solutions to
Equation~\prob{\alpha} are standing wave solutions for the
time-dependent Schrödinger Equation modeling a Bose-Einstein
condensate in a periodic optical lattice \cite{MR2502397,
  RevModPhys.78.179, 0295-5075-63-5-642, 10.1038/nature01452,
  PhysRevA.67.013602, PhysRevA.67.063608, PhysRevLett.90.160407,
  PhysRevA.66.063605, RevModPhys.71.463}.  In this model $\alpha$
is proportional to the total number of atoms in the condensate.

Set
\begin{equation}
  \label{eq:def-Sigma-alpha}
  \Sigma_\alpha\coloneqq \xlr\{\}{u\in H^1(\dR^N)\mid \abs{u}_2^2=\alpha}
\end{equation}
for $\alpha>0$.  Define the functional
$\Phi\colon H^1(\dR^N)\to\dR$ by
\begin{equation}
  \label{eq:def-functional-phi}
  \Phi(u)\coloneqq
  \frac12\int_{\dR^N}(\abs{\nabla u}^2+Vu^2)
  -\int_{\dR^N}F(u),
\end{equation}
where we have set $F(s)\coloneqq\int_0^sf$.
Then the pair $(u,\lambda)$ is a weak solution of \prob{\alpha}
if and only if $u$ is a critical point of the restriction of
$\Phi$ to $\Sigma_{\alpha}$ with Lagrange multiplier $\lambda$.

Not assuming periodicity of $V$ but instead
$\sup_{\dR^N} V = \lim _{\abs{x}\to\infty}V(x)$, the existence of
a minimizer of $\Phi$ on $\Sigma_\alpha$ in the mass-subcritical
case was shown under additional assumptions on the growth of the
nonlinearity $f$ by Lions~\cite{MR778974}; see also
\cite{MR2796242} for a different approach.  For constant $V$
solutions of \prob{\alpha} are constructed in the
mass-supercritical case in \cite{MR1430506, MR3009665,
  MR3639521}; here the corresponding critical points of
$\Phi|_{\Sigma_\alpha}$ are not local minimizers.  In
\cite{MR3638314, 2014arXiv1410.4767B, MR1948875, MR1815766} local
minimizers are found in the mass-supercritical case under
spatially constraining potentials.

The structure of the solution set of the constrained equation is
rather poorly understood up to now in the case where
$V\in L^\infty(\dR^N)$ is not constant, but $1$-periodic in all
coordinates.  In contrast, a large amount of information is
available for the \emph{free} equation
\begin{equation*}
  -\Delta u+V(x)u= f(u),\qquad        u\in H^1(\dR^N),
\end{equation*}
where essentially the parameter $\lambda$ is fixed but the
$L^2$-norm is not prescribed anymore.  Of particular interest for
us are the results on the existence of so-called \emph{multibump
  solutions}. In~\cite{MR2485422, MR2491945, MR2216902,
  MR98b:58034, MR97j:58051, sprad:1995, MR94d:35044, MR2151860,
  MR93k:35087}, an infinite number of solutions are built using
nonlinear superposition of translates of a special solution which
satisfies a nondegeneracy condition of some form.

The main goal of the present work is to apply nonlinear
superposition techniques to the constrained problem with periodic
$V$ to obtain an infinity of $L^2$-normalized solutions in the
form of multibump solutions.  We succeed in doing this, but have
to impose a stricter nondegeneracy condition than in the case of
the free equation which nevertheless is fulfilled in many
situations.  This provides, as far as we know, the first result
on multibump solutions for the constrained problem, and also the
first multiplicity result in the case of a nonconstant periodic
potential $V$. We also compute the Morse index of these
normalized multibump solutions with respect to the restricted
functional $\Phi|_{\Sigma_{\alpha}}$, and we will use the Morse
index information to derive orbital instability of the multibump
solutions.

To state our results, we need the following hypotheses. We
consider, as usual, the critical Sobolev exponent defined by
$2^*:= \frac{2N}{N-2}$ in case $N \ge 3$ and $2^*:= \infty$ in
case $N=1,2$.

\begin{enumerate}[label=(H\arabic{*}),series=hypotheses]
\item \label{item:1} $V\in L^\infty(\dR^N)$;
\item \label{item:2} $V$ is $1$-periodic in all coordinates;
\item \label{item:3} $f\in C^1(\dR)$, $f(0)=f'(0)=0$,
  \begin{equation*}
    \lim_{s\to\infty}\frac{f'(s)}{\abs{s}^{2^*-2}}=0
  \end{equation*}
  if $N\ge3$, and there is $p>2$ such that
  \begin{equation*}
    \lim_{s\to\infty}\frac{f'(s)}{\abs{s}^{p-2}}=0
  \end{equation*}
  if $N=1$ or $N=2$.
\end{enumerate}
Throughout this paper we assume \ref{item:1} and \ref{item:3}.
It is well known that $\Phi$ is well defined by
\eqref{eq:def-functional-phi} and of class $C^2$.  The standard
example for a function satisfying \ref{item:3} is
$f(s)\coloneqq\abs{s}^{p-2}s$ with $p\in(2,2^*)$. In
the following, we let $H^{-1}(\dR^N)$ denote the topological dual
of $H^1(\dR^N)$.  For our main result, we need the notion of a
{\em fully nondegenerate} critical point of
$\Phi |_{\Sigma_\alpha}$.

\begin{definition}
  \th\label{def:fully-nondegenerate}
  Assume \ref{item:1} and \ref{item:3}. For $\alpha>0$, a
  critical point $u \in H^1(\dR^N)$ of $\Phi |_{\Sigma_\alpha}$
  with Lagrangian multiplier $\lambda$ will be called {\em fully
    nondegenerate} if for every $g \in H^{-1}(\dR^N)$ there
  exists a unique weak solution $z_g \in H^1(\dR^N)$ of the
  linearized equation
  \begin{equation}
    \label{eq:determining-equation}
    -\Delta  z_g +[V-\lambda] z_g - f'(u)z_g =g  \qquad \text{in $\dR^N$,}
  \end{equation}
  and if in the case $g=u$ we have
  $\int_{\dR^N} u\, z_u  \not = 0$.  Here, as usual, we
  regard $H^1(\dR^N)$ as a subspace of $H^{-1}(\dR^N)$, so
  $u \in H^{-1}(\dR^N)$.
\end{definition}

As we shall see in Section~\ref{sec:result-about-pert} below, the
full nondegeneracy of a critical point $u \in H^1(\dR^N)$ of
$\Phi |_{\Sigma_\alpha}$ with Lagrangian multiplier $\lambda$
implies the nondegeneracy of the Hessian of
$\Phi |_{\Sigma_\alpha}$ at $u$. By definition, this Hessian is
the bilinear form
\begin{equation}
  \label{eq:defi-quadratic}
  (v,w) \mapsto \int_{\dR^N} \Bigl(\nabla v \nabla w + [V-\lambda]vw - f'(u)v w\Bigr) 
\end{equation}
defined on the tangent space
\begin{equation*}
  T_{u}\Sigma_\alpha  = \{v \in H^1(\dR^N)\mid (v,u)_2 = 0\},
\end{equation*}
see \th\ref{def:various} below.  Here
$\lr(){\cdot,\cdot}_2$ denotes the standard scalar product in
$L^2(\dR^N)$.  We also need to fix the following elementary
notation. If $n\in\dN$ and $a=(a^1,a^2,\dots,a^n)\in(\dZ^N)^n$ is
a tuple of $n$ elements from $\dZ^N$, denote
\begin{equation*}
  d(a)\coloneqq\min_{i\neq j}\abs{a^i-a^j}.
\end{equation*}
Moreover, for $b \in \dR^N$ we denote by $\cT_b$ the associated
translation operator, i.e., for $u\colon \dR^N \to \dR$ the
function $\cT_b u\colon \dR^N \to \dR$ is given by
\begin{equation*}
  \cT_b u (x):= u(x-b),\qquad \text{for $x \in \dR^N$.}
\end{equation*}

Our first main result is the following.

\begin{theorem}[Multibump Solutions]
  \th\label{thm:two-bumps} Assume \ref{item:1}--\ref{item:3} and
  fix $\alpha>0$, $n\in\dN$, $n\ge 2$. Moreover, suppose that
  $\olu$ is a fully nondegenerate critical point of
  $\Phi|_{\Sigma_{\alpha/n}}$ with Lagrangian multiplier
  $\olambda$. Then for every $\varepsilon>0$ there exists
  $R_\varepsilon>0$ such that for every $a\in(\dZ^N)^n$ with
  $d(a)\ge R_\varepsilon$ there is a critical point $u_a$ of
  $\Phi|_{\Sigma_\alpha}$ with Lagrange multiplier $\lambda_a$
  such that
  \begin{equation*}
    \biggnorm{u_a-\sum_{i=1}^n \cT_{a^i} \olu}_{H^1(\dR^N)}\le\varepsilon
    \qquad\text{and}\qquad\abs{\lambda_a-\olambda}\le\varepsilon.
  \end{equation*}
  If $\varepsilon$ is chosen small enough then $u_a$ is unique.
  Moreover, if $\olu$ is a positive function and $f(\olu) \ge 0$
  on $\dR^N$, $f(\olu) \not \equiv 0$, then $u_a$ is positive as
  well.
\end{theorem}

The proof of \th\ref{thm:two-bumps} is based on a general
Shadowing Lemma, a simple consequence of Banach's Fixed
Point Theorem, applied to approximate zeros of the gradient of
the extended Lagrangian $G_\alpha$ for the constrained
variational problem on $\Sigma_\alpha$.  If $\olu$ is a
nondegenerate local minimum of $\Phi$ on $\Sigma_{\alpha/n}$ then
it is easy to see that the sum $\wtu$ of $n$ translates of $\olu$
is an approximate zero of $\nabla G_\alpha$ if these translates
are far enough apart from each other.  The Shadowing Lemma
implies that to obtain a zero of $\nabla G_\alpha$ near $\wtu$ it
is sufficient to prove that $\rmD^2 G_\alpha(\wtu)$ is invertible
and that the norm of its inverse is bounded appropriately. This
step is the main difficulty and requires the assumption of full
nondegeneracy of $\olu$.
 
Our next result is concerned with the Morse index of the
solutions $u_a$ given in \th\ref{thm:two-bumps} with respect
to the functional $\Phi|_{\Sigma_\alpha}$. For this we recall
that the Morse index $m(u)$ of a critical point $u$ of
$\Phi|_{\Sigma_\alpha}$ with Lagrangian multiplier $\lambda$ is
defined as the maximal dimension of a subspace
$W \subset T_{u}\Sigma_\alpha$ such that the quadratic form in
\eqref{eq:defi-quadratic} is negative definite on $W$. If such
a maximal dimension does not exist, one sets $m(u)=\infty$.  We
also introduce the following additional assumption.
\begin{enumerate}[resume*=hypotheses]
\item \label{item:7} $f(s)/\abs{s}$ is nondecreasing in $\dR$ and
  $f(s)s>0$ for all $s\neq0$.
\end{enumerate}

\begin{theorem}
  \th\label{thm:two-bumps-morse-index} Assume
  \ref{item:1}--\ref{item:3}, fix $\alpha>0$, $n\in\dN$,
  $n\ge 2$, and suppose that $\olu$ is a fully nondegenerate
  critical point of $\Phi|_{\Sigma_{\alpha/n}}$ with Lagrangian
  multiplier $\olambda$ and finite Morse index
  $m(\olu)$. Moreover, let $z_\olu$ be given as in
  \th\ref{def:fully-nondegenerate} with $u=\olu$.  Then the
  critical points $u_a$ found in \th\ref{thm:two-bumps} have, for
  small $\varepsilon$, the following Morse index $m(u_a)$ with
  respect to $\Phi|_{\Sigma_\alpha}$:
  \begin{equation*}
    m(u_a) = \left \{
      \begin{aligned}
        &n (m(\olu)+1)-1 &&\qquad \text{if }
          \lr(){\olu, z_\olu}_2 < 0;\\
        &n m(\olu)  &&\qquad \text{if }\lr(){\olu, z_\olu}_2 > 0.
      \end{aligned}
    \right.
  \end{equation*}
  If moreover \ref{item:7} holds true, then $m(u_a)>0$.
\end{theorem}

The key r\^ole of the sign of the scalar product
$\lr(){\olu, z_\olu}_2$ in this theorem is not surprising since
it is closely related to variational properties of the underlying
critical point $\olu$. More precisely, we shall see in
\th\ref{lem:simple-but-imp} below that it determines the
relationship between the Morse index of $\olu$ with respect to
$\Phi|_{\Sigma_{\alpha/n}}$ and its {\em free} Morse index with
respect to the functional
$u \mapsto \Phi(u)- \olambda \abs{u}_2^2$ on $H^1(\dR^N)$.
 
We now consider the special case where \ref{item:7} holds true and
$\olu$ is a nondegenerate local minimum of
$\Phi|_{\Sigma_{\alpha/n}}$.  By a nondegenerate local minimum we
mean a critical point $\olu$ of $\Phi|_{\Sigma_{\alpha/n}}$ with
Lagrangian multiplier $\olambda$ such that the quadratic form in
\eqref{eq:defi-quadratic} is positive definite on
$T_{u}\Sigma_{\alpha/n}$.
In this case, we shall see in Section~\ref{sec:result-about-pert}
below that $\olu$ is fully nondegenerate, and we will deduce the
following corollary from \th\ref{thm:two-bumps,thm:two-bumps-morse-index} in
Section~\ref{sec:morse-index-nond}.

\begin{corollary}
  \th\label{thm:two-bumps-local-min} Assume
  \ref{item:1}--\ref{item:7} and fix $\alpha>0$, $n\in\dN$,
  $n\ge 2$. Moreover, suppose that $\olu$ is a nondegenerate
  local minimum of $\Phi|_{\Sigma_{\alpha/n}}$ with Lagrangian
  multiplier $\olambda$.  Then for every $\varepsilon>0$ there
  exists $R_\varepsilon>0$ such that for every $a\in(\dZ^N)^n$
  with $d(a)\ge R_\varepsilon$ there is a critical point
  $u_a$ of $\Phi|_{\Sigma_\alpha}$ with Lagrange multiplier
  $\lambda_a$ such that
  \begin{equation*}
    \biggnorm{u_a-\sum_{i=1}^n \cT_{a^i} \olu}_{H^1(\dR^N)}\le\varepsilon
    \qquad\text{and}\qquad\abs{\lambda_a-\olambda}\le\varepsilon.
  \end{equation*}
  If $\varepsilon$ is chosen small enough then $u_a$ is unique.
  Moreover, $u_a$ does not change sign and has Morse index
  $m(u_a) \!=\! n\!-\!1$ with respect to $\Phi|_{\Sigma_\alpha}$.
\end{corollary}

Next we present an example where the nondegeneracy hypotheses of
the previous theorems can be verified.  For this we make the
following assumptions.

\begin{enumerate}[resume*=hypotheses]
\item\label{item:4} $V\in C^2(\dR^N)$ is $1$-periodic in all
  coordinates, positive, and has a nondegenerate critical point at
  some point $x_0 \in \dR^N$.
\item\label{item:5} $f(s)= \abs{s}^{p-2}s$ for some
  $p \in (2,2^*) \ssm \{2+\frac{4}{N}\}.$
\end{enumerate}
We then consider the constrained singularly perturbed equation
\begin{equation*}
  \tag*{\prob{\alpha,\varepsilon}}
  -\varepsilon^2\Delta u+V(x)u- \abs{u}^{p-2}u=\lambda u, \qquad         u\in H^1(\dR^N), \qquad        \abs{u}_2^2=\alpha
\end{equation*}
in the semiclassical limit $\varepsilon\to0$.  Its weak solutions
correspond, for each $\varepsilon>0$, to critical points and
Lagrange multipliers of the restriction of the functional
\begin{equation*}
  \Phi_\varepsilon\colon H^1(\dR^N)\to\dR,\qquad  \Phi_\varepsilon(u)
  \coloneqq   \frac12\int_{\dR^N}(\varepsilon^2\abs{\nabla u}^2+Vu^2)
  -\frac1p\int_{\dR^N}\abs{u}^p
\end{equation*}
to $\Sigma_\alpha$. We also consider the related free problem
\begin{equation}
  \label{eq:38}\tag*{\fprob{\varepsilon}}
  -\varepsilon^2\Delta u+V(x)u= \abs{u}^{p-2}u,\qquad
  u\in H^1(\dR^N),
\end{equation}
whose weak solutions coincide with critical points of
$\Phi_\varepsilon$, for every $\varepsilon>0$.  It is well known
(see \cite{MR1956951}) that there exists a locally unique curve
of solutions of \fprob{\varepsilon} that concentrate near $x_0$
as $\varepsilon\to0$.  For our purposes we need to show
additional properties of these solutions.
\begin{theorem}
  \th\label{teo:semiclassical-existence} Assume \ref{item:4} and
  \ref{item:5}. Then there exist $\varepsilon_0>0$ and a
  continuous map $(0,\varepsilon_0)\to H^1(\dR^N)$,
  $\varepsilon\to \olu_\varepsilon$, such that the following
  properties hold true:
  \begin{enumerate}[label=\textup{(\roman*)}]
  \item \label{item:9} for each $\varepsilon\in(0,\varepsilon_0)$
    the function $\olu_\varepsilon$ is a positive solution of
    \ref{eq:38};
  \item \label{item:11} as $\varepsilon \to 0$, the functions
    $x \mapsto \olu_\varepsilon$ concentrates near $x_0$ in the
    sense that the functions $x \mapsto \olu_\eps(x_0 + \eps x)$
    converge in $H^1(\dR^N)$ to the unique radial positive
    solution $u_0 \in H^1(\dR^N)$ of the equation
    $-\Delta u_0 + V(x_0)u_0 = u_0^{p-1}$ in $\dR^N$;
  \item \label{item:12} $\abs{\olu_\varepsilon}_2^2\to0$ as
    $\varepsilon\to0$;
  \item \label{item:10} for each
    $\varepsilon\in(0,\varepsilon_0)$ the function
    $\olu_\varepsilon$ is a fully nondegenerate critical point of
    the restriction of $\Phi_{\varepsilon}$ to
    $\Sigma_{\abs{\olu_\varepsilon}_2^2}$ with Morse index
    \begin{equation}
      \label{eq:morse-index-sing-perturbed}
      m(\olu_\varepsilon) = \left \{
        \begin{aligned}
          &m_V &&\quad \text{if $2<p < 2 + \frac{4}{N}$,}\\
          &m_V+1 &&\quad \text{if $2 + \frac{4}{N}<p<2^*$.}
        \end{aligned}
      \right.
    \end{equation}
    Here $m_V$ denotes the number of negative eigenvalues of the
    Hessian of $V$ at $x_0$.
  \end{enumerate}
\end{theorem}

We emphasize that properties \ref{item:9}--\ref{item:11} were
already proved in \cite{MR1956951}, and that \ref{item:12}
follows from \ref{item:11} by a simple change of variable. For
our purposes, the property~\ref{item:10} is of key importance. We
shall also see in Section~\ref{sec:example} below that, for
$\eps \in (0,\eps_0)$,
\begin{equation}
  \label{eq:z-sing-perturbed}
  \lr(){\olu_\varepsilon, z_{\olu_\varepsilon}}_2 <0
  \text{ if $2<p < 2 + \frac{4}{N}$}\quad \text{and}\quad
  \lr(){\olu_\varepsilon, z_{\olu_\varepsilon}}_2 >0
  \text{ if $2 + \frac{4}{N}<p<2^*$,}          
\end{equation}
where $z_{\olu_\eps}$ is given as in
\th\ref{def:fully-nondegenerate} corresponding to
$u=\olu_\eps$. Since the solutions $\olu_\varepsilon$ in
\th\ref{teo:semiclassical-existence} depend continuously on
$\varepsilon$ and $\abs{\olu_\varepsilon}_2^2\to0$ as
$\varepsilon\to0$, we can find, for every $\alpha>0$ and large
enough $n\in\dN$, a number $\varepsilon_n \in (0,\eps_0)$ such
that $\abs{\olu_{\varepsilon_n}}_2^2=\alpha/n$.  The combination
of
\th\ref{thm:two-bumps,thm:two-bumps-morse-index,teo:semiclassical-existence}
with \eqref{eq:z-sing-perturbed} therefore yields the following
corollary.
\begin{corollary}
  \th\label{cor:multibumps-example} Assume \ref{item:4} and
  \ref{item:5}. Then for every $\alpha>0$ there exist
  $n_\alpha \in\dN$ and a sequence $\varepsilon_n\to0$ such that
  for every $n\ge n_\alpha$ the problem
  \prob{\alpha,\varepsilon_n} has infinitely many geometrically
  distinct positive solutions. More precisely, for every
  $n \in \dN$ with $n\ge n_\alpha$, and every $\delta>0$ there
  exists $R_{\delta,n}>0$ such that for every $a\in(\dZ^N)^n$
  with $d(a)\ge R_{\delta,n}$ there is a critical point
  $u_a$ of $\Phi_{\varepsilon_n}|_{\Sigma_\alpha}$ with Lagrange multiplier
  $\lambda_a$ such that
  \begin{equation*}
    \biggnorm{u_a-\sum_{i=1}^n \cT_{a^i} \olu_{\eps_n}}_{H^1(\dR^N)}\le \delta 
    \qquad\text{and}\qquad \abs{\lambda_a}  \le \delta.
  \end{equation*}
  If $\delta$ is chosen small enough then $u_a$ is unique.
  Moreover, $u_a$ is a positive function, and its Morse index
  with respect to $\Phi|_{\Sigma_\alpha}$ is given by
  \begin{equation*}
    m(u_a) = \left \{
      \begin{aligned}
        &n(m_V+1)-1 &&\quad \text{if $ 2<p < 2 + \frac{4}{N}$,}\\
        &n (m_V+1) &&\quad \text{if $ 2 + \frac{4}{N}<p<2^*$,}
      \end{aligned}
    \right.
  \end{equation*}
  where $m_V$ denotes the number of negative eigenvalues of the
  Hessian of $V$ at $x_0$.
\end{corollary}

Our next result is concerned with the orbital instability of the
normalized multibump solutions we have constructed in the
previous theorems. For this we focus on odd nonlinearities $f$ in
$(P_\alpha)$ satisfying \ref{item:3} and therefore assume
\begin{enumerate}[resume*=hypotheses]
\item\label{item:13} the function $f$ is odd.
\end{enumerate}
We also assume \ref{item:1} and \ref{item:3}, so $\Phi$ in
\eqref{eq:def-functional-phi} is a well defined
$C^2$-functional. If $\varphi \in \Sigma_\alpha$ is a critical
point of $\Phi |_{\Sigma_\alpha}$ with Lagrangian multiplier
$\lambda$, then the function
\begin{equation}
  \label{eq:def-solitary-wave-solution-0}
  u_\varphi\colon \dR \times \dR^N \to \dC, \qquad u_\varphi(t,x)= \varphi(x)\rme^{i\lambda t}
\end{equation}
is a solution of the time-dependent nonlinear Schrödinger
equation
\begin{equation}
  \label{eq:time-schroedinger-0}
  -i u_t =- \Delta u +V(x) u - g(\abs{u}^2)u,
\end{equation}
where $g$ is defined by $f(t)=g(\abs{t}^2)t$.  Solutions of this
special type are usually called solitary wave solutions. The
solution $u_\varphi$ is called {\em orbitally stable} if for
every $\varepsilon>0$ there exists $\delta>0$ such that every
solution $u\colon [0,t_0) \to H^1(\dR^N,\dC)$ of
\eqref{eq:time-schroedinger-0} with
$\norm{u(0,\cdot)-\varphi}_{H^1} < \delta$ can be extended to a
solution $[0,\infty) \to H^1(\dR^N,\dC)$ which satisfies
\begin{equation*}
  \sup_{0<t< \infty} \inf_{s \in \dR} \norm{u(t,\cdot)-
    u_\varphi(s,\cdot)}_{H^1} < \varepsilon.
\end{equation*}
Otherwise, $u_\varphi$ is called {\em orbitally unstable}. We
then have the following result.

\begin{theorem}
  \th\label{thm:orbital-instability}
  Assume \ref{item:1}, \ref{item:3}, and \ref{item:13}, and suppose that
  $\varphi \in \Sigma_\alpha$ is a positive function which is a
  critical point of $\Phi |_{\Sigma_\alpha}$ with positive Morse
  index and Lagrangian multiplier
  $\lambda < \inf \sigma_\rmess(-\Delta + V)$.  Then the
  corresponding solitary wave solution $u_\varphi$ of
  \eqref{eq:time-schroedinger-0} is orbitally unstable.
\end{theorem}

Here and in the following, $\sigma_\rmess(-\Delta +V)$ denotes
the essential spectrum of the Schrödinger operator $-\Delta
+V$. We note that \th\ref{thm:orbital-instability} neither
requires periodicity of $V$, nor does it require the assumption
on the oddness of a certain difference of numbers of eigenvalues
in the seminal instability result in \cite[p.~309]{MR1081647}.
\th\ref{thm:orbital-instability} applies to the normalized
multibump solutions constructed in \th\ref{thm:two-bumps} and
Corollaries~\ref{thm:two-bumps-local-min} and
\ref{cor:multibumps-example} in the case where the nonlinearity
satisfies \ref{item:7} and \ref{item:13}.  In these cases, the
extra assumption $\lambda < \inf \sigma_\rmess(-\Delta + V)$
follows from \th\ref{lem:positive-spectrum} below and the fact
that the Lagrangian multipliers of the multibump solutions are
arbitrarily close to the multiplier of the initial solution.

There are many results on the orbital stability and instability
of the standing waves generated by solutions to \prob{\alpha},
see \cite{MR2527691, MR2465996, PhysRevA.66.063605, MR901236,
  MR677997}.  However, none of these results covers the situation
addressed in \th\ref{thm:orbital-instability}.

The paper is organized as follows.  In
Section~\ref{sec:result-about-pert} we collect some preliminary
notions and observations.  In particular, here we explain our new
notions of fully nondegenerate restricted critical point and of
the free Morse index.  In Section~\ref{sec:gluing-bumps} we then
prove \th\ref{thm:two-bumps}.  In
Section~\ref{sec:morse-index-nond} we derive a general result on
the Morse index of normalized multibump solutions which gives
rise to \th\ref{thm:two-bumps-morse-index}.  At the end of
this section, we also complete the proof of
\th\ref{thm:two-bumps-local-min}.  In
Section~\ref{sec:example}, we analyze the singular perturbed
problem \ref{eq:38} and we prove
\th\ref{teo:semiclassical-existence}. In
Section~\ref{sec:orbital-instability}, we then prove the orbital
instability result given in
\th\ref{thm:orbital-instability}. Finally, in the Appendix we
provide a computation of the free Morse index of the solutions
$u_\eps$ considered in
\th\ref{teo:semiclassical-existence}. This computation is
partly contained in \cite[Proof of Theorem 2.5]{MR2403325}, but
some details have been omitted there. We therefore provide a
somewhat different argument in detail for the convenience of the
reader.

We finally remark that the main results of our paper can be
extended to more general nonlinearities.  In particular,
\th\ref{thm:two-bumps} has an abstract proof that extends to
nonlinearities that also depend on $x$, $1$-periodically in every
coordinate. This proof also extends to nonlocal nonlinearities
with convolution terms as in \cite{MR2527691}.  This follows from
Brézis-Lieb type splitting properties for these nonlinearities
that were proved in \cite{MR2216902}. 

\subsection{Notation}
\label{sec:functional-setup}

In the remainder of the paper, we write $\abs{\cdot}_p$ for the
standard $L^p(\dR^N)$-norm, $1 \le p \le \infty$. We also use the
notation $(\cdot,\cdot)_2$ for the standard $L^2(\dR^N)$-scalar
product. For the sake of brevity, we write $L^2$ in place of
$L^2(\dR^N)$ and $H^k$ in place of $H^k(\dR^N)$, for $k \in \dN$.
By \ref{item:1}, $-\Delta+V$ is a self adjoint operator in $L^2$
with domain $H^2$.  Since we assume \ref{item:1} throughout the
paper and $\lambda$ is a free parameter in \prob{\alpha}, we may
assume without loss of generality that $\gamma\coloneqq\min\sigma(-\Delta+V)>0$,
where $\sigma(-\Delta+V)$ stands for the spectrum of
$-\Delta + V$.  Then $H^1$ is the form domain (the energy space)
of $-\Delta+V$, and we may endow $H^1$ with the scalar product
\begin{equation}
  \label{eq:def-scp-H-1}
  \scp{u,v} = \int_{\dR^N} \Bigl(\nabla u \cdot \nabla v + V uv\Bigr) , \qquad u,v\in H^1.
\end{equation}
The norm $\norm{\cdot}$ induced by $\scp{\cdot,\cdot}$ is
equivalent to the standard norm on $H^1$.  It will be convenient
to denote $S\coloneqq(-\Delta+V)^{-1}$; then we have
\begin{equation}
  \label{eq:rel-scp-H-1}
  \scp{u,v} =(S^{-1/2}u,S^{-1/2}v)_2 \qquad \text{for $u,v\in H^1$.}
\end{equation}
We point out that, for a subspace $Z \subset H^1$, the notation
$Z^\perp$ always refers to the orthogonal complement of $Z$ in
$H^1$ with respect to the scalar product $\scp{\cdot,\cdot}$.

We recall that the spectrum $\sigma(-\Delta+V)$ is purely
essential if \ref{item:2} is assumed. In this case, it also
follows that all powers of $S$ are equivariant with respect to
the action of $\dZ^N$.  Hence
\begin{equation*}
  \scp{\cT_a v,\cT_a w}=\scp{v,w}
  \qquad\forall v,w\in H^1,\ \forall a\in\dZ^N.
\end{equation*}

For any two normed spaces $X,Y$ the space of bounded linear
operators from $X$ in $Y$ is denoted by $\cL(X,Y)$, and we write
$\cL(X)\coloneqq\cL(X,X)$.

For a $C^1$-functional $\Theta$ defined on $H^1$, we let
$\rmd \Theta \colon H^1 \to (H^1)^*$ denote the derivative of
$\Theta$ and $\nabla \Theta\colon H^1 \to H^1$ the gradient with
respect to the scalar product $\scp{\cdot,\cdot}$ defined in
\eqref{eq:def-scp-H-1}.  Moreover, if $\Theta$ is of class $C^2$,
then $\rmd ^2 \Theta(u)\colon H^1 \times H^1 \to \dR$ denotes the
Hessian of $\Theta$ at a point $u \in H^1$, whereas
$\rmD^2 \Theta(u) \in \cL(H^1)$ stands for the derivative of the
gradient of $\Theta$ at $u$.  We then have
\begin{equation*}
  \scp{\rmD^2 \Theta(u) v, w} = \rmd ^2 \Theta(u)[v,w] \qquad
  \text{for $v,w \in H^1$.}
\end{equation*}

\textbf{Acknowledgement:}
The authors wish to thank the referee for his/her valuable comments and corrections.

\section{Some preliminary abstract results and notions}
\label{sec:result-about-pert}

In this section we state some abstract results which will be used
in Section~\ref{sec:gluing-bumps} in the proof of
\th\ref{thm:two-bumps}.  We start with a standard corollary of
Banach's fixed point theorem, which is sometimes referred to as a {\em Shadowing Lemma}.

\begin{lemma}
  \th\label{lem:some-prel-abstr-banach}
  Let $(E,\norm{\cdot})$ be a Banach space, let ${ h}\colon E \to E$
  be continuously differentiable with derivative
  $\rmd   h\colon E \to \cL(E)$, and let $v_0 \in E$, $\delta>0$,
  $q \in (0,1)$ satisfy the following:
  \begin{enumerate}[label=\textup{(\roman*)}]
  \item $T:= \rmd   h(v_0) \in \cL(E)$ is an isomorphism.
  \item
    $\norm{ h(v_0)} <
    \frac{\delta(1-q)}{\norm{T^{-1}}_{\cL(E)}}$.
  \item $\norm{ \rmd   h(y)-T}_{\cL(E)} \le
    \frac{q}{\norm{T^{-1}}_{\cL(E)}}$ for
    $y \in B_\delta(v_0)$.
  \end{enumerate}
  Then $ h$ has a unique zero in $B_\delta(v_0)$.
\end{lemma}

The proof of this lemma is standard by showing that the map
$y \mapsto y - T^{-1} h(y)$ defines a $q$-contraction on
$\overline{B_\delta(v_0)}$. Applying Banach's fixed point theorem
to this map gives rise to a unique zero of $ h$ in
$\overline{B_\delta(v_0)}$, and it easily follows from the above
assumptions that this zero is contained in $B_\delta(v_0)$.

We will use the following immediate corollary of
\th\ref{lem:some-prel-abstr-banach}.

\begin{corollary}
  \th\label{cor:some-prel-abstr-banach}
  Let $(E,\norm{\cdot})$ be a Banach space, let $h\colon E \to E$
  be differentiable and such that its derivative
  $\rmd   h\colon E \to \cL(E)$ is uniformly continuous on bounded
  subsets of $E$. Moreover, let $( v_k)_k$ be a bounded sequence
  in $E$ such that
  \begin{enumerate}[label=\textup{(\roman*)}]
  \item $ h( v_k) \to 0$ as $k \to \infty$;
  \item $\rmd   h( v_k) \in \cL(E)$ is an isomorphism for
    $k \in \mathbb{N}$, and
    $\sup_{k \in \mathbb{N}} \norm{\rmd   h(
      v_k)^{-1}}_{\cL(E)} < \infty$.
  \end{enumerate}
  Then there exist $k_0 \in \mathbb{N}$ and $u_k \in E$,
  $k \ge k_0$, with
  \begin{equation}
    \label{eq:-corollary-banach-seq-1}
     h(u_k)=0 \qquad \text{for $k \ge k_0$}  
  \end{equation}
  and
  \begin{equation}
    \label{eq:-corollary-banach-seq-2}
    \norm{u_k-v_k} \to 0 \qquad \text{as $k \to \infty$.}
  \end{equation}
  Moreover, the sequence $(u_k)_k$ is uniquely determined by
  properties
  \eqref{eq:-corollary-banach-seq-1},~\eqref{eq:-corollary-banach-seq-2}
  for large $k$.
\end{corollary}

In the remainder of this section, we collect some preliminary
results and notions related to the functional $\Phi$ defined in
\eqref{eq:def-functional-phi} and its restrictions to spheres
with respect to the $L^2(\dR^N)$-norm.  Recall that we are
assuming conditions~\ref{item:1} and~\ref{item:3}.  We denote
\begin{equation*}
  \Psi(u)\coloneqq\int_{\dR^N}F(u),
\end{equation*}
so
\begin{equation*}
  \Phi(u)=\frac12\norm{u}^2-\Psi(u).
\end{equation*}
Following \cite{MR2216902} we say that a map $g\colon X\to Y$ of
Banach spaces $X$ and $Y$ \emph{BL-splits} if
$g(x_n)-g(x_n-x^*) \to g(x^*)$ in $Y$ if $x_n\weakto x^*$ in $X$.
For example, by \cite[Remark~3.3]{MR2216902} the maps
$\norm{\cdot}^2$ and $\abs{\cdot}_2^2$ BL-split.  The next result
about BL-splitting maps is less obvious:
\begin{lemma}\th\label{lem:psi-bl-splits}
  $\Psi$, $\nabla \Psi$ and $\rmD^2\Psi$ BL-split, and these maps
  are uniformly continuous on bounded subsets of $H^1$.
\end{lemma}

Before we give the proof we fix some $p\in(2,2^*)$ if $N\ge3$ and
we use $p$ given in \ref{item:3} if $N=1,2$.  Using \ref{item:3}
it is easy to construct, for every $\varepsilon>0$, functions
$f_{i,\varepsilon}\in C^1(\dR)$, $i=1,2,3$, and a constant
$C_\varepsilon>0$ such that
\begin{equation}\label{eq:22}
  f=\sum_{i=1}^3f_{i,\varepsilon}
\end{equation}
and such that
\begin{equation}\label{eq:26}
  \abs{f_{1,\varepsilon}'(s)}\le\varepsilon,\quad
  \abs{f_{2,\varepsilon}'(s)}\le C_\varepsilon\abs{s}^{p-2},
  \quad\text{and}\quad
  \abs{f_{3,\varepsilon}'(s)}\le\varepsilon\abs{s}^{2^*-2},
  \qquad\text{for all }s\in\dR.
\end{equation}
If $N=1,2$ we simply choose $f_{3,\varepsilon}\equiv 0$ and
ignore all terms that contain $2^*$.
\begin{proof}[Proof of \th\ref{lem:psi-bl-splits}]
  We only prove this in the case $N\ge3$; the other cases are
  treated similarly.  Consider $(u_n)\subseteq H^1$ such that
  $u_n\weakto u$.  Then $(u_n)$ is bounded in $H^1$ and therefore
  also in $L^q$ for $q\in[2,2^*]$.  For fixed $\varepsilon>0$ we
  have
  \begin{equation*}
    \abs{f'_{2,\varepsilon}(u_n)-f'_{2,\varepsilon}(u_n-u)-f'_{2,\varepsilon}(u)}_{p/(p-2)}\to0
  \end{equation*}
  by \cite[Theorem~1.3]{unicobl}.  On the other hand, there are
  varying constants $C>0$, independent of $\varepsilon$, such
  that
  \begin{equation*}
    \abs{f'_{1,\varepsilon}(u_n)-f'_{1,\varepsilon}(u_n-u)-f'_{1,\varepsilon}(u)}_\infty
    \le C\varepsilon
  \end{equation*}
  and
  \begin{equation*}
    \abs{f'_{3,\varepsilon}(u_n)-f'_{3,\varepsilon}(u_n-u)-f'_{3,\varepsilon}(u)}_{2^*/(2^*-2)}
    \le C\varepsilon
  \end{equation*}
  for all $n$.  For all $v,w\in H^1$ with $\norm{v}=\norm{w}=1$
  it follows that
  \begin{multline*}
    \bigabs{\bigscp{\biglr(){\rmD^2\Psi(u_n)-\rmD^2\Psi(u_n-u)-\rmD^2\Psi(u)}v,w}}\\
    \le C\varepsilon\abs{v}_2\abs{w}_2
    +\abs{f'_{2,\varepsilon}(u_n)-f'_{2,\varepsilon}(u_n-u)-f'_{2,\varepsilon}(u)}_{p/(p-2)}
    \abs{v}_p\abs{w}_p+C\varepsilon\abs{v}_{2^*}\abs{w}_{2^*}\\
    \le C(\varepsilon+o(1))
  \end{multline*}
  and hence
  $\limsup_{n\to\infty}\norm{\rmD^2\Psi(u_n)-\rmD^2\Psi(u_n-u)-\rmD^2\Psi(u)}
  _{\cL(H^1)}\le C\varepsilon$.  Letting $\varepsilon\to0$ we
  obtain the claim for $\rmD^2\Psi$.  The proof for the uniform
  continuity of $\rmD^2\Psi$ on bounded subsets of $H^1$ is
  similar.  Analogously, one treats the maps $\nabla\Psi$ and
  $\Psi$.
\end{proof}

We shall need the following simple consequence of assumption
\ref{item:7}.

\begin{lemma}
  \th\label{lem:h-5-consequence}
  If conditions~\ref{item:1} and \ref{item:3}--\ref{item:7} hold
  true and $u \in H^1 \ssm \{0\}$ satisfies
  $\nabla \Phi(u)= \lambda S u$ for some $\lambda \in \dR$, then
  \begin{equation*}
    \scp{(\rmD^2\Phi(u) -\lambda S)u, u} <0.
  \end{equation*}
\end{lemma}

\begin{proof}
  By \ref{item:3} and \ref{item:7}, the map
  $s\mapsto f'(s)s^2-f(s)s$ is nonnegative in $\dR$, and it is
  positive on a nonempty open subset of
  $(-\varepsilon,\varepsilon) \ssm \{0\}$ for every
  $\varepsilon>0$.  Moreover, since $u \in H^1$ is a weak
  solution of
  \begin{equation*}
    -\Delta u + [V(x)-\lambda] u = f(u) \qquad \text{in $\dR^N$}
  \end{equation*}
  by assumption, standard elliptic regularity shows that $u$ is
  continuous and that $u(x) \to 0$ as $\abs{x} \to
  \infty$. Consequently, we have
  \begin{align*}
    \scp{\rmD^2\Phi(u)u,u}-\lambda\scp{Su,u}
    &=\scp{\rmD^2\Phi(u)u,u}-\scp{\nabla\Phi(u),u}\\
    &=\scp{\nabla\Psi(u),u}-\scp{\rmD^2\Psi(u)u,u}
    =\int_{\dR^N}(f(u)u-f'(u)u^2)<0,
  \end{align*}
  as claimed.
\end{proof}

As before, for $\alpha>0$, we consider the sphere
$\Sigma_\alpha \subset H^1$ as defined in
\eqref{eq:def-Sigma-alpha}, and we let
$J_\alpha\colon\Sigma_\alpha\to\dR$ denote the restriction of
$\Phi$ to $\Sigma_\alpha$.  We note that, for
$u \in \Sigma_\alpha$, the tangent space of $\Sigma_\alpha$ at
$u$ is given by
\begin{equation}
  \label{eq:tan-space}
  T_{u}\Sigma_\alpha  = \{v \in H^1\mid (v,u)_2 = 0\} = \{v \in H^1\mid\scp{v, S u}  = 0\} \subset H^1,
\end{equation}
where latter equality follows from \eqref{eq:rel-scp-H-1}.  If
$u$ is a critical point of $J_\alpha$, we have
\begin{equation}
  \label{eq:lagrange-mult-basic}
  \nabla \Phi(u) = \lambda S u
\end{equation}
for some $\lambda \in \dR$, the corresponding Lagrange
multiplier. Moreover, the Hessian $\rmd ^2 J_\alpha (u)$ is a
well-defined quadratic form on $T_{u} \Sigma_\alpha$ given by
\begin{equation}
  \label{eq:formula-second-der-J-alpha}
  \rmd ^2 J_\alpha(u)[v,w] 
  = \scp{\rmD^2 \Phi(u)v, w} - \lambda \scp{Sv,w}
  \qquad \text{for $v, w \in T_{u} \Sigma_\alpha$.}
\end{equation}
For the general definition of the Hessian of $C^2$-functionals on
Banach manifolds at critical points, see
e.g. \cite[p. 307]{MR0158410}. To see
\eqref{eq:formula-second-der-J-alpha}, one may argue with local
coordinates for $\Sigma_\alpha$ at $u$, as is done, e.g., in
\cite[Theorem~8.9]{MR1319337} in the finite dimensional case.
Alternatively, to prove \eqref{eq:formula-second-der-J-alpha} we
may consider smooth vector fields $\tilde v$, $\tilde w$ on
$\Sigma_\alpha$ with $\tilde v(u)=v$, $\tilde w(u)=w$, and we
extend $\tilde v$, $\tilde w$ arbitrarily as smooth vector fields
$\tilde v,\tilde w\colon H^1 \to H^1$. Using
\eqref{eq:lagrange-mult-basic}, we then have
\begin{align*}
  \rmd ^2 J_\alpha(u)[v,w] &= \partial_{\tilde v} \partial_{\tilde w}
  \Phi(u)= \partial_{\tilde v} \big|_{u} \scp{\nabla \Phi, w
} = \scp{\rmD^2 \Phi(u) v,w}+
  \scp{\nabla \Phi(u),  \rmd  \tilde w(u) v}\\
  &=\scp{\rmD^2 \Phi(u) v,w} + \lambda (u , \rmd  \tilde
  w(u) v)_2 =\scp{\rmD^2 \Phi(u) v,w} - \lambda ( v ,
  w)_2,
\end{align*}
where the last equality follows from the fact that the function
$u_* \mapsto h(u_*):= (u_*, \tilde w (u_*))_2$ vanishes on
$\Sigma_\alpha$ and therefore
$0=\partial_{\tilde v} h(u)= ( v , w)_2 + (u , \rmd  w (u)
v)_2$.  

We need the following definitions.
\begin{definition}
  \th\label{def:various}
  Let $u\in H^1$ be a critical point of $J_\alpha$ with Lagrange
  multiplier $\lambda$.  Put $\Lambda\coloneqq T_u \Sigma_\alpha$ and let
  $P  \in \cL(H^1,\Lambda)$ denote the
  $\scp{\cdot,\cdot}$-orthogonal projection onto $\Lambda$.  Moreover, put
  $B\coloneqq \rmD^2 \Phi(u)-\lambda S$.
  \begin{enumerate}[label=\textup{(\alph*)}]
  \item The {\em Morse index}
    $m(u) \in \dN \cup \{0, \infty\}$ of $u$ with respect to
    $J_\alpha$ is defined as
    \begin{equation*}
      m(u):= \sup \{ \dim  Z\mid \text{$ Z$ subspace of
        $\Lambda$ with $\scp{Bv,v}<0$ for all
        $v \in  Z \ssm \{0\}$}\}.
    \end{equation*}   
  \item The {\em free Morse index}
    $m_\rmf(u) \in \dN \cup \{0, \infty\}$ of $u$ is defined as
    \begin{equation*}
      m_\rmf(u):= \sup \{ \dim  Z\mid \text{$ Z$ subspace of
        $H^1$ with $\scp{Bv,v}<0$
        for all $v \in  Z \ssm \{0\}$}\}.
    \end{equation*}   
  \item We call $u$ a {\em nondegenerate} critical point of
    $J_\alpha$ if $P  B|_\Lambda $ is an isomorphism
    of $\Lambda$.
  \item We call $u$ {\em freely nondegenerate} if $B$ is an
    isomorphism of $H^1$. In this case we put
    \begin{equation*}
      z_u:= B^{-1}S u \in H^1.
    \end{equation*}
  \end{enumerate}
\end{definition}
For a critical point $u\in H^1$ of $J_\alpha$, it is
  clear that
  \begin{equation}
    \label{eq:-cases-morse-index}
    m_\rmf(u) = m(u) \qquad \text{or}\qquad   m_\rmf(u)= m(u)+1.
  \end{equation}
In the case where $u$ is freely nondegenerate, the scalar product $(z_u,u)_2$ determines whether $u$ is nondegenerate and which case occurs in (\ref{eq:-cases-morse-index}). More precisely, we have the following simple but important lemma.  
\begin{lemma}
  \th\label{lem:simple-but-imp}
  Let $u\in H^1$ be a freely nondegenerate critical point 
of $J_\alpha$ with Lagrange multiplier $\lambda$.  
  \begin{enumerate}[label=\textup{(\alph*)}]
  \item\label{item:17} $u$ is nondegenerate if and only if $(z_u,u)_2\not =0$.
  \item\label{item:19} If $m(u)$ is finite and $(z_u,u)_2>0$, then $m_\rmf(u)=m(u)$.
  \item\label{item:20} If $m(u)$ is finite and $(z_u,u)_2<0$, then $m_\rmf(u)=m(u)+1$.
  \end{enumerate}
\end{lemma}

\begin{proof}
In the following, we let $\cN(L)$ denote the kernel and $\cR(L)$ denote the range of a linear operator $L$. Moreover, we let $B$, $P$ and $\Lambda$ be as in \th\ref{def:various}. 

\textbf{\ref{item:17}:} By definition, we have
$z_u = B^{-1} Su \in \cN(PB) \setminus \{0\}$. Moreover, we have
$\dim \cN(PB)=1$ since $B\colon H^1 \to H^1$ is an
isomorphism. Consequently,
\begin{equation*}
  \cN(PB) = \opspan(z_u) \qquad \text{and}\qquad \cR(PB)= \Lambda. 
\end{equation*}
Now, again by definition, $u$ is nondegenerate if and only if
$P B|_\Lambda\colon \Lambda \to \Lambda$ is an isomorphism, and
this holds true if and only if $H^1 = \opspan(z_u) \oplus \Lambda$.  By
(\ref{eq:tan-space}), the latter property is equivalent to
$(z_u,u)_2 \not =0$.

\textbf{\ref{item:19} and \ref{item:20}:} Since $\codim \Lambda=1$ and $z_u\notin \Lambda$, there are, for every
  $\phi\in H^1$, unique elements $\mu\in\dR$ and $w\in \Lambda$ such that
  \begin{equation}
    \label{eq:9}
    \phi=\mu z_u+w.
  \end{equation}
  Recall that $\opspan(Su)=\cN(P )=\Lambda^\bot$.  We therefore have the
  representation
  \begin{equation}
    \label{eq:10}
    \begin{aligned}
    \scp{B\phi,\phi}
      &=\mu^2\scp{Bz_u,z_u}+2\mu\scp{Bz_u,w}+\scp{Bw,w}\\
      &=\mu^2\scp{Su,z_u}+2\mu\scp{Su,w}+\scp{Bw,w}\\
      &=\mu^2(z_u,u)_2+\scp{Bw,w}.
    \end{aligned}
  \end{equation}
  To see \ref{item:19}, recall that the definition of $m(u)$
  implies the existence of a subspace $Z \subset \Lambda$ of
  codimension $m(u)$ in $\Lambda$ such that
  $\scp{B \phi, \phi} \ge 0$ for all $\phi \in
  Z$. Since $z_u \notin \Lambda$, the space
  $\widetilde Z:= \opspan(z_u) \oplus Z $ has at most codimension $m(u)$
  in $H^1$. Moreover, in the representation \eqref{eq:9} for
  $\phi\in \widetilde Z$ we find $w \in Z$.  Therefore,
  \eqref{eq:10} yields $\scp{B\phi,\phi} \ge\scp{Bw,w}\ge0$.
  This implies $m_\rmf(u) \le m(u)$, and thus equality follows by
  \eqref{eq:-cases-morse-index}.

  To see \ref{item:20}, let $Z \subset \Lambda$ be an
  $m(u)$-dimensional subspace such that
  $\scp{B w, w}<0$ for all $w \in Z \ssm
  \{0\}$. Put $\widetilde Z:= \opspan(z_u) \oplus Z$. Then
  $\dim \widetilde Z = m(u)+1$, and for the representation
  \eqref{eq:9} for $\phi\in \widetilde Z \ssm \{0\}$ we find
  $w \in Z$.  Then \eqref{eq:10} implies $\scp{B \phi, \phi} < 0$
  since either $\mu \not = 0$ or $w \in Z \ssm
  \{0\}$. Consequently, $m_\rmf(u) \ge m(u)+1$, and thus equality
  follows by \eqref{eq:-cases-morse-index}.
\end{proof}

Parts \ref{item:19} and \ref{item:20} of \th\ref{lem:simple-but-imp} can also be derived from  \cite[(2.7) of
Theorem~2]{MR772868}, see also \cite{MR957686}.  For the
convenience of the reader we gave a simple direct proof.

\begin{definition}
  \th\label{def:fully-nondegenerate-2} 
A critical point $u\in H^1$ of $J_\alpha$ will be called 
{\em fully
    nondegenerate} if $u$ is freely nondegenerate and 
the equivalent properties in \th\ref{lem:simple-but-imp}\ref{item:17} hold true.\end{definition}

\th\ref{def:fully-nondegenerate-2} is consistent with
\th\ref{def:fully-nondegenerate}, as the function $z_u = B^{-1} S u$ defined in \th\ref{def:various} is uniquely determined as the weak solution of \eqref{eq:determining-equation} with $g=u$.

In the next lemma, we show that nondegenerate local minima of $J_\alpha$ are fully nondegenerate critical points.

\begin{lemma}
  \th\label{lem:simple-but-imp-1}
Suppose that \ref{item:7} holds true, and let $u\in H^1$ be a nondegenerate critical point 
of $J_\alpha$ with $m(u)=0$ (i.e., $u$ is a nondegenerate
    local minimum of $J_\alpha$). Then $u$ is fully nondegenerate, and either $u$ or $-u$ is a positive function.
\end{lemma}

\begin{proof}
We continue using the notation from the proof of \th\ref{lem:simple-but-imp}.  
Since $u$ is nondegenerate, we have $\Lambda=\cR(P B|_\Lambda)$ and therefore 
  $H^1=\cN(P)+ \cR(B|_\Lambda)$.  This implies
  $\codim \cR(B) \le \codim \cR(B|_\Lambda)\le 1$ and hence
  that $\cR(B)$ is closed. Since
  $P B|_\Lambda$ is injective, $\cN(B)\cap \Lambda=\{0\}$ and hence
  $\dim\cN(B)\le 1$.  If $\dim\cN(B)=1$ were true, then we would
  have $H^1=\cN(B)\oplus \Lambda$.  Since the quadratic form
  $\scp{B\cdot,\cdot}$ is positive definite on $\Lambda$ it would be
  positive semidefinite on $H^1$, in contradiction with
  \th\ref{lem:h-5-consequence}.  Therefore $\cN(B)=\{0\}$ and $B$,
  being symmetric with closed range, is an isomorphism. Hence $u$ is freely nondegenerate, and thus it is also fully nondegenerate.  

  Next, we suppose by contradiction that $u$ changes sign. A
  variant of the proof of \th\ref{lem:h-5-consequence} then shows
  that the quadratic form $\scp{B\cdot,\cdot}$ is negative
  definite on the two-dimensional subspace
  $\opspan(u^+,u^-) \subset H^1$, where $u^\pm\coloneqq\max\{0,\pm u\}$
  denotes the positive, respectively negative part of $u$. Since
  this space has a nontrivial intersection with $\Lambda$, we
  thus obtain a contradiction to the assumption $m(u)=0$.
\end{proof}

Next we add an observation for the case where
$u$ is a fully nondegenerate critical point of
$J_\alpha$ and a {\em positive} function.

\begin{lemma}
  \th\label{lem:positive-spectrum} Let $u \in H^1$ be a fully
  nondegenerate critical point of $J_\alpha$ with Lagrangian
  multiplier $\lambda$ such that $u$ is a positive function and
  $f(u) \ge 0$ on $\dR^N$, $f(u) \not \equiv 0$. Then we
  have
  \begin{equation}
    \label{eq:lambda-spectrum}
    \lambda < \inf \sigma(-\Delta + V).  
  \end{equation}
\end{lemma}

\begin{proof}
  Since $u$ is freely nondegenerate, we see that
  \begin{equation}
\label{eq:persson-0}   
 \lambda \not \in \sigma(-\Delta +V -f'(u)).
  \end{equation}
  Moreover, $u(x) \to 0$ as $|x| \to \infty$ by standard elliptic
  estimates, and the same is true for the functions
  $x \mapsto f'(u(x)), \: x \mapsto
  \frac{f(u(x))}{u(x)}$. Consequently, by (\ref{eq:persson-0}),
  Theorem~14.6 and the proof of Theorem~14.9 in \cite{MR1361167}
  we have for $L_0\coloneqq-\Delta+V$ and
  $L\coloneqq-\Delta+V-\frac{f(u)}{u}$ that
  \begin{equation*}
    \lambda \notin \sigma_\rmess(-\Delta +V -f'(u))
    =\sigma_\rmess(L_0)
    = \sigma_\rmess(L),
  \end{equation*} 
  where $\sigma_\rmess$ denotes the essential spectrum. Since $u$
  is an eigenfunction of the Schrödinger operator $L$
  corresponding to the eigenvalue $\lambda$, it follows that
  $\lambda$ is isolated in $\sigma(L)$.  Since moreover $u$ is
  positive, it is then easy to see that
  $\lambda= \inf \sigma(L)$, and that $\lambda$ is a simple
  eigenvalue.  On the other hand, the assumption $\frac{f(u)}{u}
  \ge 0$ implies that
  \begin{equation*}
    \inf\sigma(L_0)\ge\inf\sigma(L)=\lambda.
  \end{equation*}
  If $\lambda=\inf\sigma(L_0)$ were true, we could obtain
  from $\lambda\notin\sigma_\rmess(L_0)$ that $\lambda$ is
  also an isolated eigenvalue of $L_0$ with a positive
  eigenfunction $v$.  But then, since $f(u) \not \equiv 0$ by assumption,
  \begin{equation*}
    \lambda
    =\frac{\int_{\dR^N}(\abs{\nabla v}^2+Vv^2)}{\int_{\dR^N}v^2}
    >\frac{\int_{\dR^N}(\abs{\nabla
        v}^2+(V-f(u)/u)v^2)}{\int_{\dR^N}v^2}
    \ge\lambda,
  \end{equation*}
  a contradiction.  Hence $\lambda<\inf\sigma(L_0)$.
\end{proof}

We close this section by introducing the {\em extended Lagrangian} 
\begin{equation*}
  G_\alpha\colon H^1\times\dR\to\dR, \qquad G_\alpha(u,\lambda
  )\coloneqq\Phi(u)-\frac{\lambda}{2} \xlr(){\abs{u}_2^2-\alpha} =
  \Phi(u)-\frac{\lambda}{2} \xlr(){\scp{S u,u} - \alpha}.
\end{equation*}
By definition, $u\in H^1$ is a critical point of $J_\alpha$ with Lagrange
multiplier $\lambda$ if and only if $(u,\lambda )$ is a critical
point of $G_\alpha$. We endow $H^1\times\dR$ with the natural
scalar product
\begin{equation*}
  \scp{(u,s),(v,t)}\coloneqq\scp{u,v}+st.
\end{equation*}
The respective gradient of $G_\alpha$ is
\begin{equation}\label{eq:8}
  \nabla G_\alpha\colon H^1 \times \dR \to H^1 \times \dR,  
  \qquad  \nabla G_\alpha(u,\lambda)=\xlr(){\nabla\Phi(u)- \lambda
    Su,-\frac{1}{2}\xlr(){\abs{u}^2_2- \alpha}}.
\end{equation}
Moreover, we have
\begin{equation}\label{eq:7}
  \rmD^2 G_\alpha(u,\lambda)[(v,\mu)]=\biglr(){\rmD^2\Phi(u)v-\lambda
    S v-\mu Su, -\scp{Su,v}}.
\end{equation}

The operator $\rmD^2G_\alpha(u,\lambda)$ is known in the
literature as the \emph{Bordered Hessian} of $\Phi$ at
$(u,\lambda)$.  It has been used extensively in finite
dimensional settings to discern local extrema of restricted
functionals, see, e.g., \cite{MR1791915, MR759215, MR1321457,
  MR1145891, MR1237226, MR1465784}. We will use it only in
Section~\ref{sec:gluing-bumps} below for a gluing procedure
respecting an $L^2$-constraint.

Although we do not need this property in the present paper, we
note that a critical point $u \in H^1$ of $J_\alpha$ is
nondegenerate if and only if $\rmD^2G_\alpha(u,\lambda)$ is an
isomorphism of $H^1\times\dR$. The proof is straightforward.

\section{Gluing Bumps with $L^2$-Constraint}
\label{sec:gluing-bumps}
This section is devoted to the proof of
\th\ref{thm:two-bumps}, which we reformulate in the
following way for matters of convenience. We continue to use the
notation introduced in Section~\ref{sec:result-about-pert}.

\begin{theorem}
  \th\label{thm:two-bumps-general} Assume \ref{item:1}--\ref{item:3}
  and fix $\alpha>0$.  Given $n\in\dN$, $n\ge 2$, suppose that
  $\olu$ is a fully nondegenerate critical point of
  $J_{\alpha/n}$ with Lagrange multiplier
  $\olambda$.  Let also $(a_k)\subseteq(\dZ^N)^n$ be a sequence
  such that $d(a_k)\to\infty$ as $k\to\infty$. Then there exists
  $k_0 \in \dN$ such that for $k \ge k_0$ there exist critical
  points $u_k$ of $J_{\alpha}$ with Lagrange
  multiplier $\lambda_k$. Moreover, we have
  \begin{equation}\label{eq:24-bump-general}
    \norm{u_k-v_k} \to 0 
    \quad\text{and}\quad\abs{\lambda_k-\olambda} \to 0\quad
    \text{as $k \to \infty$,}\qquad \text{where }  v_k := \sum_{i=1}^n \cT_{a^i_k} \olu \in H^1, 
  \end{equation}
  and the sequence $(u_k)_k$ is uniquely determined by these
  properties for large $k$.  Furthermore, if $\olu$ is a positive
  function and $f(\olu) \ge 0$ on $\dR^N$,
  $f(\olu) \not \equiv 0$, then $u_k$ is positive as well for
  large $k$.
\end{theorem}

The remainder of this section is devoted to the proof of this
theorem. Let $\alpha>0$, $n\ge2$, and $\olu$, $\bar\lambda$ be as
in the statement of the theorem.  Since $\olu$ is nondegenerate
and freely nondegenerate, \th\ref{def:various} and
\th\ref{def:fully-nondegenerate-2} imply that
\begin{equation}
  \label{eq:def-B}
  \text{$B:=\rmD^2 \Phi(\olu)-\olambda  S  \in  \cL(H^1)$ is an isomorphism}
\end{equation}
and that
\begin{equation}
  \label{eq:exist-z-u}
  \text{there exists $z_{\olu} \in H^1$ with $(z_\olu,\olu)_2 \not = 0$ and $B z_{\olu} = S \olu$.}
\end{equation}
Let $(a_k)\subseteq(\dZ^N)^n$ be a sequence such that
$d(a_k)\to\infty$ as $k\to\infty$, and let $v_k \in H^1$ be given
as in \eqref{eq:24-bump-general} for $k \in \dN$.  For simplicity
we assume that
\begin{equation}
  a^1_k=0 \qquad\text{for all }k\in\dN.\label{eq:51}
\end{equation}
We wish to prove that
\begin{equation}
  \label{eq:3}
  \nabla G_{\alpha}(v_k,\olambda )\to 0\qquad\text{as } k\to\infty
\end{equation}
and that
\begin{equation}\label{eq:5}
  \parbox{.8\linewidth}{$\rmD^2 G_{\alpha}(v_k,\olambda ) \in \cL(H^1\times\dR)$
    is invertible for large $k$,
    and the norm of the inverse remains bounded as $k\to\infty$.}
\end{equation}
Once these assertions are proved, we may apply
\th\ref{cor:some-prel-abstr-banach} with
$ h \coloneqq\nabla G_\alpha$ to find, for $k$ large, critical
points $u_k$ of $J_\alpha$ with Lagrange multiplier
$\lambda_k$ such that \eqref{eq:24-bump-general} holds true.
Here we use the fact that the sequence $(v_k)_k$ is bounded in
$H^1$ and that $\rmD^2 \Phi$ is uniformly continuous on bounded
subsets of $H^1$.

By the BL-splitting properties, \eqref{eq:8} implies
\begin{equation*}
  \biggnorm{\nabla G_{\alpha}(v_k,\olambda )
    -\sum_{i=1}^n\nabla G_{\alpha/n}(\cT_{a^i_k}\olu,\olambda )}_{\cL(H^1 \times \dR)}
  \to0.
\end{equation*}
Since
$\norm{\nabla G_{\alpha/n}(\cT_{a^i_k}\olu,\olambda )}_{\cL(H^1
  \times \dR)} = \norm{\nabla G_{\alpha/n}(\olu,\olambda
  )}_{\cL(H^1 \times \dR)} = 0$ for $i=1,2,\dots,n$ and every
$k$, \eqref{eq:3} follows.

We now turn to the (more difficult) proof of \eqref{eq:5}.  For
this we consider the operators
\begin{equation*}
  B_k:=\rmD^2\Phi(v_k)-\olambda  S \in \cL(H^1)
  \qquad \text{for $k \in \dN$.} 
\end{equation*}
and we claim that
\begin{equation}
  \label{eq:50}
  \cT_{-a^i_k} B_k\cT_{a^i_k} w \to B w \quad \text{in $H^1$ for $w \in H^1$, $i=1,2,\dots,n$}.
\end{equation}
To see this, we recall that $\rmD^2\Psi$ BL-splits and that
therefore
\begin{equation}
  \label{eq:61}
  \rmD^2\Psi(v_k)=\sum_{j=1}^n\rmD^2\Psi(\cT_{a^j_k} \olu)+o(1)\qquad \text{in $\cL(H^1)$,}
\end{equation}
which implies that
\begin{equation}
  \label{eq:equality-B_k}
  B_k= I-\olambda S - \rmD^2\Psi(v_k) = I-\olambda S- \sum_{j=1}^n\rmD^2\Psi(\cT_{a^j_k} \olu)+o(1)\qquad \text{in $\cL(H^1)$.}
\end{equation}
It is easy to see that
\begin{equation}
  \label{eq:17}
  \cT_{-a^i_k}\rmD^2\Psi(\cT_{a^i_k} \bar
  u)\cT_{a^i_k}=\rmD^2\Psi(\olu)
  \qquad\text{for $k \in \dN$ and $i=1,\dots,n$.}
\end{equation}
Moreover, if $i\neq j$, then for $w \in H^1$ we have
\begin{equation}
  \label{eq:62}
  \rmD^2\Psi(\cT_{a^j_k} \olu)\cT_{a^i_k} w = \cT_{a^j_k} \cT_{-a^j_k}\rmD^2\Psi(\cT_{a^j_k} \olu)\cT_{a^j_k}  \cT_{a^i_k-a^j_k} w= 
  \cT_{a^j_k}\rmD^2\Psi(\olu) \cT_{a^i_k-a^j_k} w    \to 0 
\end{equation}
in $H^1$, since $\cT_{a^i_k-a^j_k} w \rightharpoonup 0$ and
$\rmD^2\Psi(\olu) \in \cL(H^1)$ is a compact operator. Combining
\eqref{eq:equality-B_k}--\eqref{eq:62} and recalling that $S$
commutes with $\cT_{a^i_k}$, we find that
\begin{align*}
  \cT_{-a^i_k} B_k\cT_{a^i_k} w &= (I-\olambda S)w - \sum_{j=1}^n  \cT_{-a^i_k} \rmD^2\Psi(\cT_{a^j_k} \olu) \cT_{a^i_k}w  +o(1)\\
  &= (I-\olambda S)w -\rmD^2\Psi(\olu) w+o(1) = B w+o(1) \qquad
  \text{as $k \to \infty$}
\end{align*}
for $w \in H^1$ and $i=1,\dots,n$, as claimed in \eqref{eq:50}.

We note that \eqref{eq:50} implies that
\begin{equation}
  \label{eq:50-cor}
  \cT_{-a^i_k} B_k \cT_{a^j_k} w = \cT_{a^j_k-a^i_k} \cT_{-a^j_k}  B_k \cT_{a^j_k} w = \cT_{a^j_k-a^i_k} B w +o(1) \rightharpoonup 0  
  \quad \text{in $H^1$}
\end{equation}
for $w \in H^1$ and $i \not = j$.  We now prove \eqref{eq:5} by
contradiction. Supposing that \eqref{eq:5} does not hold true, we
find, after passing to a subsequence, that there are $w_k\in H^1$
and $\mu_k\in\dR$ such that $\norm{w_k}^2+\mu_k^2=1$ and
$\rmD^2 G_\alpha(v_k,\olambda )[(w_k,\mu_k)]\to0$. By
\eqref{eq:7} this implies
\begin{equation}
  \label{eq:4}
  B_k w_k-\mu_k Sv_k \to 0  \qquad \text{in $H^1$}
\end{equation}
and
\begin{equation}
  \label{eq:6}
  (v_k,w_k)_2 \to 0 \qquad \text{in $\dR$.}
\end{equation}
Define for $i=1,2,\dots, n$, possibly after passing to a
subsequence, the functions
\begin{equation*}
  w^i\coloneqq \wlim_{k\to\infty} \cT_{-a^i_k} w_k \in H^1
\end{equation*}
and $\mu:= \lim _{k \to \infty} \mu_k$. Let
$z_\olu \in H^1$ be given as in \eqref{eq:exist-z-u}. Forming the
$H^1$-scalar product of \eqref{eq:4} with $\cT_{a^i_k} z_{\olu}$
and using \eqref{eq:50} together with the fact that
$\cT_{-a^i_k} v_k \rightharpoonup \olu$ in $H^1$, we obtain that
\begin{align*}
  o(1) &=\scp{ B_k w_k ,\cT_{a^i_k} z_\olu}
  -\mu_k\scp{Sv_k ,\cT_{a^i_k} z_\olu}= \scp{   w_k , B_k \cT_{a^i_k} z_\olu}  - \mu_k (v_k , \cT_{a^i_k} z_\olu)_2\\
  &=\scp{ \cT_{-a^i_k} w_k, \cT_{-a^i_k} B_k \cT_{a^i_k} z_\olu
  }-\mu_k (\cT_{-a^i_k} v_k, z_\olu)_2 =
  \scp{w^i,  B  z_\olu }- \mu (\olu,  z_\olu)_2 + o(1)\\
  &= \scp{w^i, S \olu }- \mu (\olu, z_\olu)_2 +o(1)= (w^i,
  \olu)_2- \mu (\olu, z_\olu)_2 +o(1)
\end{align*}
for $i=1,\dots,n$. Hence
\begin{equation*}
  (w^i,  \olu)_2= \mu (\olu,  z_\olu)_2 \qquad \text{for $i=1,\dots,n$.}
\end{equation*}
By \eqref{eq:6} we thus have that
\begin{equation*}
  0 =\lim_{k \to \infty}(v_k,w_k)_2 = \lim_{k \to
    \infty}\sum_{i=1}^n (\cT_{a^i_k} \olu, w_k)_2 = \lim_{k \to
    \infty} \sum_{i=1}^n ( \olu, \cT_{-a^i_k} w_k)_2 = \sum_{i=1}^n
  ( \olu, w^i)_2 = n \mu (\olu, z_\olu)_2.
\end{equation*}
Since $(\olu, z_\olu)_2 \not =0$, this gives $\mu = 0$. Hence
\eqref{eq:4} reduces to
\begin{equation}
  \label{eq:4-reduced}
  B_k w_k \to 0 \qquad \text{in $H^1$ as $k \to \infty$.}
\end{equation}
We now set
\begin{equation*}
  z_k \coloneqq w_k-\sum_{j=1}^n \cT_{a^j_k} w^j\qquad \text{for
    $k \in \dN$,}
\end{equation*}
so
\begin{equation}
  \label{eq:55}
  \cT_{-a^i_k} z_k\weakto0 \qquad \text{for $i=1,\dots,n.$}
\end{equation}
By \eqref{eq:50},~\eqref{eq:50-cor} and \eqref{eq:4-reduced} we
have
\begin{equation}
  \begin{aligned}
    0 = \wlim_{k \to \infty} \cT_{-a^i_k} B_k w_k &=  \wlim_{k \to \infty} \Bigl[ \sum_{j=1}^n  \cT_{-a^i_k} B_k \cT_{a^j_k} w^j +   \cT_{-a^i_k} B_k z_k\Bigr] \\
    &= B w^i + \wlim_{k \to \infty} \cT_{-a^i_k} B_k z_k.
  \end{aligned}
  \label{eq:4-reduced-1}
\end{equation}
Moreover,
\begin{equation}
  \label{eq:z_k-olu-strongly}
  \rmD^2 \Psi(\olu) \cT_{-a^i_k} z_k  \to 0 \qquad \text{in $H^1$ for $i=1,\dots,n$}
\end{equation}
by \eqref{eq:55} and since $\rmD^2 \Psi(\olu) \in \cL(H^1)$ is a
compact operator, which by \eqref{eq:17} implies that
\begin{equation}
  \label{eq:z_k-olu-strongly-1}
  \cT_{-a^i_k} \rmD^2 \Psi(\cT_{a^j_k}\olu) z_k =  \cT_{a^j_k-a^i_k}\rmD^2 \Psi(\olu) \cT_{-a^j_k} z_k \to 0 \qquad \text{in $H^1$}
\end{equation}
for $i,j=1,\dots,n$. Using \eqref{eq:equality-B_k} again, we obtain
\begin{align*}
  \wlim_{k \to \infty} \cT_{-a^i_k} B_k z_k &= \wlim_{k \to \infty}  \Bigl(\cT_{-a^i_k} (I-\olambda S) z_k-   \sum_{j=1}^n\cT_{-a^i_k} \rmD^2 \Psi(\cT_{a^j_k}\olu) z_k\Bigr)\\
  &= \wlim_{k \to \infty} (I-\olambda S) \cT_{-a^i_k} z_k = 0
\end{align*}
for $i=1,\dots,n$. Combining this with \eqref{eq:4-reduced-1}, we
conclude that $B w^i =0$ for $i=1,\dots,n$ and thus
\begin{equation*}
  w^i = 0 \qquad \text{for $i=1,\dots,n$}
\end{equation*}
by \eqref{eq:def-B}. We therefore have $w_k= z_k$ for all
$k$. Recalling \eqref{eq:4-reduced}, \eqref{eq:equality-B_k},
\eqref{eq:51}, and choosing $i=1$ in \eqref{eq:z_k-olu-strongly}
and \eqref{eq:z_k-olu-strongly-1}, we find
\begin{align*}
  o(1) = B_k w_k=B_k z_k &=(I-\olambda S) z_k  -\sum_{j=1}^n \rmD^2 \Psi(\cT_{a^j_k}\olu) z_k +o(1)=  (I- \olambda S)z_k +o(1)\\
  &= (I- \olambda S)z_k -\rmD^2 \Psi(\olu) z_k+o(1) = B z_k+o(1)=B
  w_k + o(1).
\end{align*}
and thus $w_k \to 0$ in $H^1$ by \eqref{eq:def-B}.  Since
$\mu=0$, this contradicts our assumption that
$\norm{w_k}^2+\mu_k^2=1$ for all $k$.  This proves \eqref{eq:5},
as desired.

In the following we assume $N\ge3$.  The cases $N=1,2$ are proved
similarly, ignoring those terms below that include the critical
exponent $2^*$.

As remarked above, applying \th\ref{cor:some-prel-abstr-banach}
with $ h \coloneqq\nabla G_\alpha$ now yields, for $k$ large,
critical points $u_k$ of $J_\alpha$ with Lagrange
multiplier $\lambda_k$ such that \eqref{eq:24-bump-general} holds
true. To finish the proof of \th\ref{thm:two-bumps-general}, we
now assume that $\olu \in H^1$ is positive with $f(\olu) \ge 0$
in $\dR^N$, $f(\olu) \not \equiv 0$, and we show that $u_k$ is
also positive for $k$ large. By \th\ref{lem:positive-spectrum} we
then have $\olambda < \inf \sigma(-\Delta + V) =\gamma$, so
\begin{equation*}
  \int_{\dR^N} \Bigl( \abs{\nabla v}^2 + [V-\olambda] \abs{v}^2 \Bigr)
  \ge (\gamma-\olambda) \norm{v}^2 \qquad \text{for all $v \in H^1$.}
\end{equation*}
On the other hand, for fixed $\varepsilon\in (0,\gamma-\olambda)$ it
easily follows from \ref{item:3}, Sobolev embeddings, the
representation \eqref{eq:22}, and \eqref{eq:26}, that there is a
constant $C>0$ such that
\begin{equation*}
  \int_{\dR^N} f(v) v
  \le \varepsilon\norm{v}^2+ C \norm{v}^{p} +\varepsilon\norm{v}^{2^*} \qquad \text{for
    $v \in H^1$.}
\end{equation*}
Moreover, since $v_k$ is positive, \eqref{eq:24-bump-general}
implies that $u_k^- := \min \{u_k,0\} \to 0$ in $H^1$ as
$k \to \infty$. However, we have
\begin{align*}
  0 &=  \int_{\dR^N}\Bigl(-\Delta u_k + [V- \lambda_k]u_k- f(u_k)\Bigr)u_k^-   \\
  &= \int_{\dR^N} \Bigl( \abs{\nabla u_k^-}^2 + [V-\lambda_k]
  \abs{u_k^-}^2 \Bigr) - \int_{\dR^N} f(u_k^-)u^-_k
\end{align*}
and therefore
\begin{align*}
  (\gamma-\olambda)\norm{u_k^-}^2
    &\le \int_{\dR^N} \Bigl( \abs{\nabla u_k^-}^2 + [V-\olambda]
    \abs{u_k^-}^2 \Bigr)\\
    &=o(1) \abs{u_k^-}_2^2 +
    \int_{\dR^N} \Bigl( \abs{\nabla
      u_k^-}^2 + [V-\lambda_k] \abs{u_k^-}^2 \Bigr)\\
    &= o(1)\norm{u_k^-}^2 + \int_{\dR^N} f(u_k^-)u^-_k\\
    &\le (\varepsilon+o(1))\norm{u_k^-}^2 + C \norm{u_k^-}^p
    + \varepsilon \norm{u_k^-}^{2^*}.
\end{align*}
By the choice of $\varepsilon$, this implies that $u_k^- = 0$ for
large $k$. Consequently, $u_k$ is strictly positive on $\dR^N$
for large $k$ by the strong maximum principle. The proof of
\th\ref{thm:two-bumps-general} is finished.

\section{Morse Index and nondegeneracy of normalized multibump
  solutions}
\label{sec:morse-index-nond}
In this section, we prove  a general result on the
nondegeneracy and the Morse index of normalized multibump
solutions built from fully nondegenerate critical points of the
restriction of $\Phi$ to $\Sigma_{\alpha/n}$. Moreover, we also
complete the proof of \th\ref{thm:two-bumps-local-min} at
the end of the section.
 
Recall, for $\alpha>0$ and a critical point $u$ of
$J_\alpha = \Phi |_{\Sigma_{\alpha}}$, the definitions of the Morse index
$m(u)$ and the free Morse index $m_\rmf(u)$ given in
\th\ref{def:various}.  The following theorem is
the main result of this section, and together with
\th\ref{lem:simple-but-imp} it readily implies
\th\ref{thm:two-bumps-morse-index}.
 
\begin{theorem}
  \th\label{thm:morse-index-general}
  Assume \ref{item:1}--\ref{item:3} and fix $\alpha>0$.  Given
  $n\in\dN$, $n\ge 2$, suppose that $\olu$ is a fully
  nondegenerate critical point of $J_{\alpha/n}$
  with Lagrange multiplier $\olambda$ and finite Morse index
  $m(\olu)$. Furthermore, let $(a_k)\subseteq(\dZ^N)^n$ be a
  sequence such that $d(a_k)\to\infty$ as $k\to\infty$, and
  such that the critical points $u_k$ of
  $J_\alpha$ with Lagrange multiplier $\lambda_k$
  and with
  \begin{equation}\label{eq:24-morse-index}
    \norm{u_k-v_k} \to 0 
    \quad\text{and}\quad\abs{\lambda_k-\olambda} \to 0\quad
    \text{as $k \to \infty$,}\qquad \text{where }  v_k := \sum_{i=1}^n \cT_{a^i_k} \olu  \in  H^1
  \end{equation}
  from \th\ref{thm:two-bumps-general} exist for all $k$.  Then,
  for $k$ sufficiently large, $u_k$ is a nondegenerate critical
  point of $J_\alpha$, $m(u_k) = n(m(\olu)+1)-1$ if
  $\lr(){\olu,z_\olu}_2<0$, and $m(u_k)=nm(\olu)$ if
  $\lr(){\olu,z_\olu}_2>0$.  If \ref{item:7} holds true, then
  $m(u_k)>0$ for large $k$.
\end{theorem}

To prove this Theorem, we set $B:= \rmD^2 \Phi(\olu)-\olambda S$
and $B_k \coloneqq \rmD^2\Phi(v_k)-\olambda S$, as in
Section~\ref{sec:gluing-bumps}.  Moreover, we consider the self
adjoint operators
\begin{equation*}
  C_k\coloneqq \rmD^2\Phi(u_k)-\lambda_k
  S  \in\cL(H^1)
\end{equation*}
for $k\in\dN$.  First we show that the constrained critical
points $u_k$ of $\Phi$ are freely nondegenerate and that
\begin{equation*}
  m_\rmf(u_k)=nm_\rmf(\olu)\qquad\text{for large } k.
\end{equation*}
To this end it is sufficient to prove the following

\begin{lemma}
  \th\label{lem:free-morse-index-multi} It holds true that
  \begin{align}
    \limsup_{k\to\infty}\adjustlimits\inf_{\substack{W\leqslant H^1\\\dim W =n
        m_\rmf(\olu)\:}}
    \sup_{\substack{w\in W\\\norm{w}=1}}
    \scp{C_kw,w}
    &<0\\
    \shortintertext{and}
    \liminf_{k\to\infty}\adjustlimits\inf_{\substack{W\leqslant
        H^1\\\dim W = nm_\rmf(\olu)+1\:}}
    \sup_{\substack{w\in W\\\norm{w}=1}}
    \scp{C_kw,w}
    &>0.
  \end{align}
\end{lemma}
\begin{proof}
By \eqref{eq:24-morse-index} and since
$\rmD^2 \Phi\colon H^1 \to \cL(H^1)$ is uniformly continuous on
bounded subsets of $H^1$, the assertion follows once we have established the following estimates:
\begin{align}
  \limsup_{k\to\infty}\adjustlimits\inf_{\substack{W\leqslant H^1\\\dim W =n
      m_\rmf(\olu)\:}}
  \sup_{\substack{w\in W\\\norm{w}=1}}
  \scp{B_kw,w}
  &<0, \label{eq:13}\\
  \liminf_{k\to\infty}\adjustlimits\inf_{\substack{W\leqslant
      H^1\\\dim W = nm_\rmf(\olu)+1\:}}
  \sup_{\substack{w\in W\\\norm{w}=1}}
  \scp{B_kw,w}
  &>0.\label{eq:14}
\end{align}
Let $Z \subset H^1$ denote the generalized eigenspace of the
  self-adjoint operator $B$ in $H^1$ corresponding to its
  $m_\rmf(\olu)$ negative eigenvalues.  Pick $\delta>0$ such that
  $\scp{B w,w} \le -\delta \norm{w}^2$ for all $w \in Z$ and
  $\scp{B y,y} \ge \delta \norm{y}^2$ for all $y \in Z^\perp$.
  Put
  \begin{equation*}
    Z_k\coloneqq \sum_{i=1}^n  \cT_{a^i_k} Z \subset H^1 \qquad \text{for $k \in \dN$.}
  \end{equation*}
  Since $d(a_k)\to\infty$, the sum is direct and hence
  $\dim Z_k = nm_\rmf(\olu)$ for $k$ sufficiently large.  If
  $w_k\in Z_k$ satisfies $\norm{w_k}=1$ for all $k$, then it
  suffices to show
  \begin{equation}
    \limsup_{k\to\infty} \scp{B_k w_k,w_k} \le
    -\delta\label{eq:12}
  \end{equation}
  along a subsequence to prove \eqref{eq:13}.  We write
  \begin{equation*}
    w_k = \sum_{i=1}^n \cT_{a^i_k} \rho_k^i    \qquad \text{for $k \in \dN$ with $\rho_k^i \in Z$.} 
  \end{equation*}
  Since $Z$ is finite dimensional, we may pass to a subsequence
  such that $\rho_k^i \to \rho^i \in Z$ for $i=1,\dots, n$ as
  $k \to \infty$.  It is easy to see that then
  \begin{equation*}
    1= \norm{w_k}^2 = \sum_{i=1}^n \norm{\rho^i}^2 + o(1)\qquad \text{as
      $k \to \infty$.}
  \end{equation*}
  Thus \eqref{eq:50} and \eqref{eq:50-cor} imply that
  \begin{align*}
    \scp{B_k w_k,w_k} = \sum_{i,j=1}^n \scp{ B_k \cT_{a^i_k}
      \rho_k^i, \cT_{a^j_k} \rho_k^j }&= \sum_{i,j=1}^n
    \scp{\cT_{-a^j_k} B_k \cT_{a^i_k} \rho^i, \rho^j }+o(1)
    = \sum_{i=1}^n \scp{B \rho^i,  \rho^i} +o(1)\\
    & \le -\delta \sum_{i=1}^n \norm{\rho^i}^2 + o(1) = -\delta +
    o(1),
  \end{align*}
  that is, \eqref{eq:12}.

  If $y_k\in Z_k^\perp$ satisfies $\norm{y_k}=1$ for all $k$,
  then it suffices to show
  \begin{equation}
    \liminf_{k\to\infty} \scp{B_k y_k,y_k} \ge
    \delta\label{eq:15}
  \end{equation}
  for a subsequence to prove \eqref{eq:14}.  Passing to a
  subsequence, we may assume that
  \begin{equation*}
    w^i\coloneqq\wlim_{k\to\infty}\cT_{-a_k^i} y_k 
  \end{equation*}
  exists for $i=1,\dots,n$. Let $v \in Z$. Since
  $\cT_{a_k^i} v \in Z_k$, we infer that
  \begin{equation*}
    0
    = \scp{\cT_{a_k^i} v, y_k}
    = \scp{v, \cT_{-a_k^i} y_k}
    = \scp{v,w^i} +o(1)
    \qquad \text{for $i=1,\dots,n$.}
  \end{equation*}
  Consequently,
  \begin{equation}
    \label{eq:claim-w-sup-i}
    w^i \in  Z^\perp  \qquad \text{for $i = 1,\dots,n$.} 
  \end{equation}
  We now set
  \begin{equation*}
    z_k\coloneqq y_k-\sum_{i=1}^n \cT_{a_k^i} w^i \qquad \text{for $k \in \dN$,}
  \end{equation*}
  noting that
  \begin{equation}
    \label{eq:zero-weak-lim-z-k-1}
    \wlim_{k\to\infty}\cT_{-a_k^i} z_k=0 \qquad \text{for $i=1,\dots,n$.}
  \end{equation}
  In particular, this implies that
  \begin{equation}
    \label{eq:zero-weak-lim-z-k-2}
    z_k\weakto 0 \qquad \text{in $H^1$}
  \end{equation}
  by \eqref{eq:51} which we may again assume without loss of
  generality.  Using \eqref{eq:50}, \eqref{eq:50-cor}, and
  \eqref{eq:zero-weak-lim-z-k-1} we obtain the splitting
  \begin{equation}
    \begin{aligned}
      \scp{B_k y_k,y_k}
      &=     \scp{B_k z_k, z_k}+2 \sum_{i=1}^n \scp{ B_k \cT_{a_k^i} w^i, z_k}+ \sum_{i,j=1}^n \scp{ B_k \cT_{a_k^i} w^i, \cT_{a_k^j} w^j}\\
      &=     \scp{B_k z_k, z_k}+2 \sum_{i=1}^n \scp{\cT_{-a_k^i} B_k \cT_{a_k^i} w^i, \cT_{-a_k^i} z_k}+ \sum_{i,j=1}^n \scp{\cT_{-a_k^j} B_k \cT_{a_k^i} w^i, w^j} \\
      &= \scp{B_k z_k, z_k}+ \sum_{i=1}^n \scp{B w^i,w^i}+o(1),
    \end{aligned}
    \label{eq:morse-splitting-1}
  \end{equation}
  where
  \begin{equation}
    \begin{aligned}
      \scp{B_k z_k, z_k}
      & = \norm{z_k}^2- \lambda \abs{z_k}_2^2 - \scp{\rmD^2 \Psi(v_k) z_k,z_k}\\
      & = \norm{z_k}^2- \lambda \abs{z_k}_2^2 - \sum_{i=1}^n \scp{\rmD^2 \Psi(\cT_{a^i_k} \olu) z_k,z_k} +o(1) \\
      & = \norm{z_k}^2- \lambda \abs{z_k}_2^2
      - \sum_{i=1}^n \scp{\rmD^2 \Psi(\olu) \cT_{-a^i_k} z_k, \cT_{-a^i_k} z_k}+o(1) \\
      & = \norm{z_k}^2- \lambda \abs{z_k}_2^2 +o(1)\\
      & = \norm{z_k}^2- \lambda \abs{z_k}_2^2 - \scp{\rmD^2 \Psi(\olu) z_k, z_k} + o(1)  \\
      &= \scp{B z_k, z_k}+o(1).
    \end{aligned}
    \label{eq:morse-splitting-2}
  \end{equation}
  Here we have used \eqref{eq:61}, \eqref{eq:17}, \eqref{eq:zero-weak-lim-z-k-1},
  \eqref{eq:zero-weak-lim-z-k-2}, and the compactness of the
  operator $\rmD^2 \Psi(\olu) \in \cL(H^1)$.

  Let $P \in \cL(\cH^1)$ denote the
  $\scp{\cdot,\cdot}$-orthogonal projection on $Z$, and let
  $Q:= I-P$. Since $P$ has finite range, we see that
  \begin{equation}
    \label{eq:z_k-lambda-proj}
    z_k - Q z_k = P z_k  \to 0 \qquad \text{in $H^1$  as $k \to \infty$.}  
  \end{equation}
  Combining \eqref{eq:claim-w-sup-i},
  \eqref{eq:morse-splitting-1}, \eqref{eq:morse-splitting-2}, and
  \eqref{eq:z_k-lambda-proj}, we obtain
  \begin{align*}
    \scp{B_k y_k,y_k} &= \scp{B Q z_k, Q z_k} +  \sum_{i=1}^n \scp{B w^i,w^i}+o(1) \ge \delta \Bigl( \norm{Q z_k}^2 + \sum_{i=1}^n \norm{w^i}^2\Bigr)+o(1)\\
    &= \delta \Bigl( \norm{z_k}^2 + \sum_{i=1}^n
    \norm{w^i}^2\Bigr)+o(1)= \delta \norm{y_k}^2 +o(1)= \delta +
    o(1),
  \end{align*}
  and hence \eqref{eq:15}.
\end{proof}

From \th\ref{lem:free-morse-index-multi} it follows that $C_k$ is
invertible for large $k$ and that the norm of its inverse remains
bounded as $k\to\infty$.  We now recall the function
$z_{u_k}= C_k^{-1} S u_k \in H^1$, which by
\th\ref{lem:simple-but-imp} is of key importance to compute
$m(u_k)$.

\begin{lemma}
  \label{z-u-k-weak-lim}
  For $i=1,\dots,n$ we have that
  \begin{equation*}
    \cT_{-a^i_k}z_{u_k} \weakto z_{\olu} =B^{-1}S \olu \quad \text{in $H^1$ as $k \to \infty$.}
  \end{equation*}
\end{lemma}

\begin{proof}
  Let $\psi \in H^1$, and let $\varphi = B^{-1} \psi \in
  H^1$.  Recalling that $\rmD^2 \Phi\colon H^1 \to \cL(H^1)$ is
  uniformly continuous on bounded subsets of $H^1$, we may deduce
  from (\ref{eq:50}) that
\begin{equation*}
  \cT_{-a^i_k} C_k \cT_{a^i_k} \varphi
  = \cT_{-a^i_k} B_k \cT_{a^i_k} \varphi + o(1) \to B \varphi
  = \psi \qquad \text{in $H^1$}
\end{equation*}
as $k \to \infty$. Since moreover the sequence $(z_{u_k})_k$ is bounded in $H^1$ and 
$\cT_{-a^i_k} u_k \weakto \olu$ in $H^1$ as $k \to \infty$, we have that 
\begin{align*}
  \scp{ z_{\olu}, \psi }
  &=\scp{ B^{-1} (S \olu),  \psi }
  = \scp{ S \olu,  \varphi }
  = \scp{ S (\cT_{-a^i_k} u_k),  \varphi } + o(1)
  = \scp{ S u_k, \cT_{a^i_k} \varphi } + o(1)\\
  &= \scp{ C_k z_{u_k}, \cT_{a^i_k} \varphi } + o(1)
  = \scp{ z_{u_k}, C_k \cT_{a^i_k} \varphi } + o(1)
  = \scp{ \cT_{-a^i_k}z_{u_k}, \cT_{-a^i_k} C_k \cT_{a^i_k} \varphi } + o(1)\\
  &=\scp{ \cT_{-a^i_k}z_{u_k}, \psi }+o(1)
  \qquad \text{as $k \to \infty$.}
\end{align*}
\end{proof}

We may now complete the 

\begin{proof}[Proof of \th\ref{thm:morse-index-general}]
With the help of Lemma~\ref{z-u-k-weak-lim}, we compute
\begin{align*}
  \lr(){u_k,z_{u_k}}_2=\lr(){v_k,z_{u_k}}_2+o(1)= \sum_{i=1}^n \lr(){\cT_{a^i_k} \olu, z_{u_k}}_2+o(1)&= \sum_{i=1}^n \lr(){\olu, \cT_{-a^i_k} z_{u_k}}_2+o(1)\\
&=n\lr(){\olu,z_\olu}_2+o(1).
\end{align*}
Since $\lr(){\olu,z_\olu}_2 \not= 0$ as $\olu$ is fully nondegenerate by assumption, we infer that $\lr(){u_k,z_{u_k}}_2$ is also nonzero and has the same sign as $\lr(){\olu,z_\olu}_2$ for large $k$. Moreover, $u_k$ is freely
nondegenerate by
\th\ref{lem:free-morse-index-multi}, so \th\ref{lem:simple-but-imp} yields that $u_k$ is a fully nondegenerate critical point of $\Phi\big|_{\Sigma_\alpha}$ for large $k$.  Its Morse index is, by the same
token, $m(u_k) =m_\rmf(u_k)-1 =nm_\rmf(\olu)-1 =n(m(\olu)+1)-1$
if $\lr(){\olu,z_\olu}_2<0$, and it is
$m(u_k) =m_\rmf(u_k) =nm_\rmf(\olu) =nm(\olu)$ if
$\lr(){\olu,z_\olu}_2>0$.

To show the last statement of the present theorem, suppose that
\ref{item:7} is satisfied.  \th\ref{lem:h-5-consequence} implies
that $\scp{B\olu,\olu}<0$, that is, $m_\rmf(\olu)>0$.  In any
case it follows from the preceding calculations that $m(u_k)>0$
for large $k$.  This completes the proof of
\th\ref{thm:morse-index-general}.
\end{proof}

We close this section by completing the 

\begin{proof}[Proof of \th\ref{thm:two-bumps-local-min}]
  Let $\olu$ be a nondegenerate local minimum of
  $J_{\alpha/n}$ with Lagrange multiplier
  $\olambda$. Moreover, let $(a_k)\subseteq(\dZ^N)^n$ be a
  sequence such that $d(a_k)\to\infty$ as $k\to\infty$. By
  \th\ref{lem:simple-but-imp-1}, $\olu$ is fully
  nondegenerate and, without loss of generality, a positive
  function. Thus, \ref{item:7} and \th\ref{thm:two-bumps-general}
  imply the existence of positive critical points $u_k$ of
  $J_\alpha$ with Lagrange multiplier $\lambda_k$
  for large $k$ and such that \eqref{eq:24-morse-index} holds
  true. Moreover, the sequence $(u_k)_k$ is uniquely determined
  by these properties.  Since moreover $m_\rmf(\olu)>0=m(\olu)$ by
  \ref{item:7} and \th\ref{lem:h-5-consequence},
  \th\ref{thm:morse-index-general} now implies that $u_k$ is
  nondegenerate with $m(u_k) = n-1$ for large $k$.
\end{proof}

\section{Proof of \th\ref{teo:semiclassical-existence}}
\label{sec:example}

In this section we wish to prove
\th\ref{teo:semiclassical-existence}. For this we will assume
hypotheses \ref{item:4} and \ref{item:5}.  Without loss of
generality we may also assume for the nondegenerate critical
point $x_0$ of $V$ that
\begin{equation*}
  x_0 =0\qquad \text{and}\qquad V(x_0)=1.  
\end{equation*}
We are then concerned with positive solutions of the singularly
perturbed equation
\begin{equation}
  \label{eq:grossi-1}
  -\varepsilon^2 \Delta u+V(x)u = \abs{u}^{p-2}u, \qquad         u\in H^1,
\end{equation}
where $p \in (2,2^*)$. By \cite[Theorem 1.1]{MR1956951}, there
exists $\varepsilon_0$ and a family of positive single peak
solutions $\olu_\eps$, $\varepsilon \in (0,\varepsilon_0)$, of
\eqref{eq:grossi-1} which concentrates at $x_0=0$. This means
that each $\olu_\eps$ has only one local maximum, and the
rescaled functions
\begin{equation}
  \label{def-u-eps}
  u_\eps \in H^1, \qquad u_\eps(x) \coloneqq \olu_\varepsilon(\eps x)
\end{equation}
converge, as $\eps \to 0$, in $H^1$ to the unique radial positive solution
of the limit equation
\begin{equation}
  \label{eq:limit-equation}
  -\Delta u_0 + u_0 = u_0^{p-1} \qquad \text{in $\dR^N$.}
\end{equation}
Moreover, as follows from the uniqueness statement in
\cite[Theorem 1.1]{MR1956951}, this convergence property after
rescaling determines the solutions $\olu_\eps$ uniquely for
$\eps>0$ small. In addition, we can assume by \cite[Theorem
6.2]{MR1956951} that $\olu_\eps$ is nondegenerate, i.e., the linear operator
\begin{equation}
  \label{eq:nondegeneracy-grossi-olu}
 H^1 \mapsto H^{1},\qquad v \mapsto v - (p-1)(-\eps^2 \Delta + V)^{-1}\olu_\eps^{p-2}v \qquad \text{is an isomorphism}
\end{equation}
for $\eps\in(0,\varepsilon_0)$. Here, for $\eps >0$, the operator $-\eps^2\Delta + V \in \cL(H^{1},H^{-1})$ is understood as the Hilbert space isomorphism $H^1\to H^{-1}$ associated with the scalar product 
$$
(u,v) \mapsto  \int \limits_{\dR^N}(\eps^2 \nabla u\cdot\nabla v+V uv)
$$ 
 on $H^1$ via Riesz's representation theorem. Since $0<\min V\le\max V<\infty$, this scalar product is equivalent to the 
standard scalar product on $H^1$, which we denote by 
\begin{equation}
\label{standard-scalar-product}
  \scp{u,v}_{H^1}\coloneqq\int_{\dR^N}(\nabla u\cdot\nabla v+uv)
\end{equation}
 We also let $\norm{\cdot}_{H^1}$ denote 
the associated norm.

\begin{lemma}
  \th\label{lem:continuous-dependance} The map
  $(0,\varepsilon_0)\to H^1$, $\varepsilon\mapsto \olu_\varepsilon$
  is continuous.
\end{lemma}

\begin{proof}
For $\eps>0$, let $K(\eps):= -\eps^2\Delta + V \in \cL(H^{1},H^{-1})$. Then the map $K\colon (0,\infty)\to\cL(H^1,H^{-1})$ is continuous. Moreover, since $p$ is subcritical, the nonlinear superposition operator $H^1 \to H^{-1}$, $u \mapsto |u|^{p-2}u$ is of class $C^1$. Consequently, 
  the map
$$
h\colon(0,\infty)\times H^1 \to H^{-1}, \quad \qquad
    (\varepsilon,u) \mapsto
    K(\varepsilon) u-\abs{u}^{p-2}u
$$
  is continuous, and continuously differentiable in its second
  argument. Since $\olu_\eps$ is a weak solution of (\ref{eq:grossi-1}), we have $h(\varepsilon,u_\varepsilon)=0$. Furthermore, the operator 
$$
h_u(\varepsilon,u_\varepsilon) = K(\varepsilon) -(p-1)\abs{\olu_\eps}^{p-2} \in \cL(H^1,H^{-1})
$$ 
is an isomorphism as a consequence of (\ref{eq:nondegeneracy-grossi-olu}).  Hence the claim follows from the implicit function theorem, see, e.g.,
  \cite[Theorem~15.1]{MR787404}.
\end{proof}

Since the map $\eps \mapsto \olu_\eps$ is continuous and
\begin{equation*}
  \abs{\olu_\eps}_2^2 = \int_{\dR^N} \olu^2_\eps = \eps^N
  \int_{\dR^N}u_\eps^2 = \eps^N \int_{\dR^N}u_0^2 + o(1)=
  o(1) \qquad \text{as $\eps \to 0$,}
\end{equation*}
the assertions \ref{item:9}--\ref{item:12} of
\th\ref{teo:semiclassical-existence} are already
verified. The remainder of this section is devoted to the proof
of \th\ref{teo:semiclassical-existence}\ref{item:10}.

For this we first note that the function $u_\eps \in H^1$ defined in (\ref{def-u-eps}) satisfies the rescaled equation
\begin{equation}
  \label{eq:66-1}
  -\Delta u_\eps +V_\eps(x)u_\eps = \abs{u_\eps}^{p-2}u_\eps, \qquad         u\in H^1
\end{equation}  
with
\begin{equation}
  \label{eq:def-Veps}
  V_\eps\colon \dR^N \to \dR, \qquad V_\eps(x)= V(\eps x).
\end{equation}
Moreover, by (\ref{eq:nondegeneracy-grossi-olu}), the linear operator
\begin{equation}
  \label{eq:nondegeneracy-grossi}
  B^\eps \in \cL(H^1), \qquad B^\eps v = v - (p-1)(-\Delta + V_\eps)^{-1}u_\eps^{p-2}v \qquad \text{is an isomorphism}
\end{equation}
for $\eps\in(0,\varepsilon_0)$.
We also note that the functions $u_\eps$ have uniform
exponential decay, i.e., there exist constants $\alpha,C>0$ such
that
\begin{equation}
  \label{eq:uniform-exp}
  \abs{u_\eps(x)} \le C \rme^{-\alpha \abs{x}} \qquad \text{for all $x \in \dR^N$, $\eps \in (0,\eps_0)$,}   
\end{equation}
see \cite[Lemma 4.2.(i)]{MR1956951}. Moreover,
\begin{equation}
  \label{eq:H2-conv-unif}
  u_\eps \to u_0 \qquad \text{in $H^2(\dR^N)$ and uniformly in $\dR^N$,}  
\end{equation}
see \cite[Theorem 4.1 and Lemma 4.2(ii)]{MR1956951}.  Note that
$u_\varepsilon$ satisfies \cite[Equation~(4.1)]{MR1956951} with
$c_{i,y,\varepsilon}=0$ since it is a solution of
\eqref{eq:66-1}.

We need to recall some properties of the unique radial positive
solution $u_0$ of the limit equation \eqref{eq:limit-equation}
and therefore consider the functional
\begin{equation*}
  \Phi_0^*\colon H^1\to\dR,\qquad  \Phi_0^*(u)\coloneqq   \frac12\int_{\dR^N}(\abs{\nabla u}^2+u^2)
  -\frac{1}{p}\int_{\dR^N}\abs{u}^p,
\end{equation*}
It is easy to see that $\rmD^2\Phi_0^*(u_0) \in \cL(H^1)$ has
exactly one negative eigenvalue, the value $2-p$, with
corresponding eigenspace generated by $u_0$.  Here, the symbol
$\rmD^2$ denotes the derivative of the gradient with respect to
the scalar product $\scp{\cdot,\cdot}_{H^1}$.

Its kernel is
spanned by the partial derivatives
$\partial_1{u_0},\partial_2{u_0},\dots,\partial_N{u_0}$, see
\cite[Lemma~4.2(i)]{MR94h:35072}.  Letting $\tilde H$ denote the
$\scp{\cdot,\cdot}$-orthogonal complement of
$\opspan(\partial_1{u_0},\partial_2{u_0},\dots,\partial_N{u_0})$ in
$H^1$, we therefore find that the operator
\begin{equation*}
  B^0 \in \cL(H^1), \qquad B^0 v = \rmD^2 \Phi_0^*({u_0}) v = v -
  (p-1)[\Delta+1]^{-1} {u_0}^{p-2}
\end{equation*} 
restricts to an isomorphism $\tilde H \to \tilde H$. Moreover,
$\tilde H$ contains all radial functions, so in particular
$u_*:= [\Delta+1]^{-1}{u_0} \in \tilde H$. Consequently, there
exists a unique ${z_*} \in \tilde H$ with $B^0 {z_*} = u_*$.

\begin{lemma}
  \th\label{lem:criterion}
  We have
  \begin{equation*}
    ({z_*},u_0)_2= \Bigl(\frac{N}{4}-\frac{1}{p-2}\Bigr)\abs{u_0}_2^2 =
    \frac{p-(2+4/N)}{4N(p-2)}\abs{u_0}_2^2.
  \end{equation*}
\end{lemma}

\begin{proof}
  For $\lambda>0$, consider the function
  \begin{equation*}
    w_\lambda \in H^1,\qquad w_\lambda(x)
    = \lambda^{\frac{1}{p-2}} u_0 (\sqrt{\lambda}x)\qquad \text{for $x \in \dR^N$,}
  \end{equation*}
  which is the unique radial positive solution of
  \begin{equation}
    \label{eq-tobias:1}
    -\Delta w_\lambda  + \lambda w_\lambda - w_\lambda^{p-1} = 0    
    \qquad \text{in $\dR^N$,}
  \end{equation}
  so $w_1 = u_0$. Moreover, consider
  \begin{equation*}
    \tilde z \in H^1, \qquad \tilde
    z(x)=\frac{\partial}{\partial \lambda}\Big|_{\lambda=
      1}w_\lambda(x).
  \end{equation*}
  We claim that $z_*= -\tilde z$. Indeed, we have
  $B^0 \tilde z = -u_*$ since differentiating \eqref{eq-tobias:1}
  at $\lambda=1$ yields
  \begin{equation}
    \label{eq-tobias:2}
    -\Delta \tilde z  +\tilde z - (p-1) u_0^{p-2}\tilde z
    =-u_0   \qquad \text{in $\dR^N$.}
  \end{equation}
  Moreover, $\tilde z \in \tilde H$ since $\tilde z$ is a radial
  function. By the remarks above, this implies that
  $z_*= -\tilde z$. We therefore compute that
  \begin{align*}
    (z_*,u_0)_2
    &= - (\tilde z,u_0)_2
    = - \frac12\frac{\rmd}{\rmd\lambda}\Big|_{\lambda=1}
    \abs{w_\lambda}_2^2
    =  - \frac12\frac{\rmd}{\rmd\lambda}\Big|_{\lambda=1}\Bigl(
    \lambda^{\frac{2}{p-2}}\int_{\dR^N} u_0^2(\sqrt{\lambda}x)\dint x \Bigr)\\
    & = - \frac12\frac{\rmd}{\rmd\lambda}\Big|_{\lambda=1}
    \lambda^{\frac{2}{p-2}-\frac{N}{2}}\abs{u_0}_2^2
    =\frac12\Bigl(\frac{N}{2}-\frac{2}{p-2}\Bigr)\abs{u_0}_2^2,
  \end{align*}
  as claimed.
\end{proof}
Next we collect some properties of the scaled potentials
$V_\varepsilon$, $\eps \in (0,\eps_0)$ defined in
\eqref{eq:def-Veps}. Note that these functions are uniformly
bounded and satisfy
\begin{equation}
  \label{eq:Veps-limit-1}
  \abs{V_\eps(x)-1} \le c\, \eps^2\abs{x}^2
  \qquad \text{for $x \in \dR^N$, $\eps \in (0,\eps_0)$, with a constant $c>0$.}
\end{equation}
We also note that
\begin{equation}
  \label{eq:Veps-limit-2}
  \lim_{\eps \to 0} \frac{\partial_i V_\eps(x)}{\eps^2} =\sum_{j=1}^N \partial_{ij} V(0) x_j 
  \qquad \text{locally uniformly in $x \in \dR^N$}
\end{equation}
for $i=1,\dots,N$, so
\begin{equation}
  \label{eq:Veps-limit-2-1}
  \abs{\partial_i V_\eps(x)} \le c\, \eps^2 \abs{x}
  \qquad \text{for $x \in \dR^N$, $\eps \in (0,\eps_0)$, with a constant $c>0$.}
\end{equation}
Next we consider
\begin{equation*}
  z_\eps:= [B^{\eps}]^{-1}(-\Delta + V_\eps)^{-1}u_\eps \in
  H^1, \qquad \text{for $\eps \in (0,\eps_0)$,}
\end{equation*}
where $B^\eps$ is defined in \eqref{eq:nondegeneracy-grossi}.
Hence $z_\eps$ is the unique weak solution of
\begin{equation}\label{eq:11}
  -\Delta z_\eps + V_\varepsilon(x) z_\eps -(p-1)u_\eps^{p-2}z_\eps
  = u_\eps \qquad \text{in $\dR^N$.}
\end{equation}
We claim that
\begin{equation}
  \label{eq:limit-scalar-product-relation}
  (z_\eps, u_\eps)_2 \to (z_*,u_0)_2 \qquad \text{as $\eps \to 0$.}  
\end{equation}
To prove this, we argue by contradiction and suppose that there
exists $\delta>0$ and a sequence $(\eps_n)_n \in (0,\eps_0)$ such
that $\eps_n \to 0$ as $n \to \infty$ and
\begin{equation}
  \label{eq:limit-scalar-product-relation-contra}
  \abs{(z_n,u_n)_2 - (z_*,w)_2} \ge \delta  \qquad \text{for all $n \in \dN$, where $z_n:= z_{\eps_n}$ and $u_n:= u_{\eps_n}.$}  
\end{equation}
We first claim that the sequence $(z_n)_n$ is bounded in
$H^1$. Indeed, if not, we can pass to a subsequence such that
$\norm{z_n}_{H^1}>0$ for all $n$ and $\norm{z_n}_{H^1} \to \infty$ as
$n \to \infty$. We then consider $y_n:= \frac{z_n}{\norm{z_n}_{H^1}}$,
and we may pass to a subsequence such that $y_n \weakto y$ in
$H^1$.  Since $y_n$ is a weak solution of the equation
\begin{equation}
  \label{eq:weak-eq-y-n}
  -\Delta y_n + V_{\varepsilon_n} y_n   -(p-1)u_n^{p-2}y_n = \frac{u_n}{\norm{z_n}_{H^1}}  \qquad \text{in $\dR^N$ for every $n$,}
\end{equation}
we have
\begin{align*}
  \int_{\dR^N} \Bigl[\nabla y \nabla v + y v-  (p-1)u_0^{p-2} v\Bigr] &= \lim_{n \to \infty} \int_{\dR^N} \Bigl[\nabla y_n \nabla v + V_{\eps_n} y_n v-  (p-1)u_n^{p-2} y_n v\Bigr]\\
  &= \lim_{n \to \infty} \frac{1}{\norm{z_n}_{H^1}} \int_{\dR^N} u_n
  v = 0 \qquad \text{for every $v \in H^1$.}
\end{align*}
Consequently, $y \in H^1$ is a weak solution of
$-\Delta y +y -(p-1)u_0^{p-2}y=0$ in $\dR^N$, which means that
$B^0 y = 0$. Hence there exist $a_1,\dots,a_N \in \dR$
with $y = \sum _{i=1}^N a_i \partial_i u_0$.  Next we
note that $\partial_i u_n$ solves the equation
\begin{equation*}
  -\Delta (\partial_i u_n) + V_\eps \partial_i u_n + u_n \partial_i
  V_{\eps_n} - (p-1)u_n^{p-2}\partial_i u_n= 0 \qquad \text{for
    $i=1,\dots,N$.}
\end{equation*}
Multiplying this equation with $y_n$ and integrating over
$\dR^N$, we obtain by \eqref{eq:weak-eq-y-n} that
\begin{equation*}
  \int_{\dR^N} u_n \, y_n \partial_i V_{\eps_n}  =-
  \frac{1}{\norm{z_n}_{H^1}} \int_{\dR^N}u_n \partial_i u_n  = 0 \qquad
  \text{for all $n \in \dN$.}
\end{equation*}
Dividing this equation by $\eps_n^2$ and passing to the limit, we
may then use \eqref{eq:uniform-exp}, \eqref{eq:Veps-limit-2},
\eqref{eq:Veps-limit-2-1} and Lebegue's Theorem to see that
\begin{align*}
  0 &= \lim_{n \to \infty}\frac{1}{\eps_n^2} \int_{\dR^N} u_n\, y_n \partial_i V_{\eps_n}    = \sum_{j=1}^N \int_{\dR^N}  \partial_{ij} V(0) x_j u_0(x) y(x) \dint x\\
  &=\sum_{\ell, j=1}^N a_\ell \partial_{ij} V(0) \int_{\dR^N} x_j
  u_0(x) \partial_{\ell} u_0(x) \dint x =-
  \frac{\abs{u_0}_2^2}{2} \sum_{j=1}^N a_j \partial_{ij} V(0)
  \quad \text{for $i=1,\dots,N$.}
\end{align*}
Here we have integrated by parts in the last step. Since $0$ is a
nondegenerate critical point of $V$ by assumption, we conclude
that $a_j=0$ for $j=1,\dots,N$ and therefore $y=0$.  This implies
in particular that $(y_n^2)$ is bounded in $L^{p/2}$ and that
$y_n^2\to0$ in $L^{p/2}_\loc$.  Moreover,
$u_n^{p-2}\to u_0^{p-2}$ in $L^{p/(p-2)}$.  Testing
\eqref{eq:weak-eq-y-n} with $y_n$ we obtain that
\begin{equation*}
  \int_{\dR^N} \bigl(\abs{\nabla y_n}^2 + V_{\varepsilon_n}\abs{y_n}^2
  \bigr)= (p-1) \int_{\dR^N}u_n^{p-2}\abs{y_n}^2 +
  \frac{1}{\norm{z_n}_{H^1}} \int_{\dR^N}u_n y_n  \to 0
\end{equation*}
as $n \to \infty$ and therefore $\norm{y_n}_{H^1} \to 0$ as
$n \to \infty$, which is a contradiction. We thus conclude that
the sequence $(z_n)_n$ is bounded.  We may thus pass to a
subsequence such that $z_n \weakto z$ in $H^1$. We then have by
\eqref{eq:11}
\begin{align*}
  \int_{\dR^N} \Bigl[\nabla z \nabla v + z v-  (p-1)u_0^{p-2} v\Bigr] &= \lim_{n \to \infty} \int_{\dR^N} \Bigl[\nabla z_n \nabla v + V_{\eps_n} z_n v-  (p-1)u_n^{p-2} z_n v\Bigr]\\
  &= \lim_{n \to \infty} \int_{\dR^N} u_n v = \int_{\dR^N}
  u_0 v \qquad \text{for every $v \in H^1$.}
\end{align*}
Consequently, $z \in H^1$ is a weak solution of
$-\Delta z +z -(p-1)u_0^{p-2}z=u_0$ in $\dR^N$, which means that
$B^0 z = u_*$. As a consequence, $B^0(z-z_*)=0$, which implies
that
$z-z_* \in \opspan(\partial_1 u_0,\partial_2u_0,\dots,\partial_N u_0)$ and
therefore $(z-z_*,u_0)_2 = 0$. We thus conclude that
\begin{equation*}
  (z_n,u_n)_2 \to (z,u_0)_2 = (z_*,u_0)_2 \qquad \text{as
    $n \to \infty$,}
\end{equation*}
contrary to \eqref{eq:limit-scalar-product-relation-contra}. This
shows \eqref{eq:limit-scalar-product-relation}, as
claimed. Combining \eqref{eq:limit-scalar-product-relation} with
\th\ref{lem:criterion}, we see that for fixed
$p \in (2,2^*) \ssm \{2+\frac{4}{N}\}$, we may take
$\eps_0>0$ smaller if necessary such that
\begin{equation}
  \label{eq:z-sing-perturbed-rescaled}
  (z_\eps,u_\eps)_2 <0 \quad \text{if $2<p < 2 + \frac{4}{N}$}\qquad \text{and}\qquad (z_\eps,u_\eps)_2 >0 \quad \text{if $2 + \frac{4}{N}<p<2^*$.}          
\end{equation}
Moreover, from \eqref{eq:z-sing-perturbed-rescaled} we
immediately deduce \eqref{eq:z-sing-perturbed} by rescaling.
Since $\olu_\varepsilon$ is a critical point of
$\Phi_\varepsilon$, it is also a critical point of
$\Phi_\varepsilon|_{\sum_{\abs{\olu_\varepsilon}_2^2}}$ with
Lagrange multiplier $0$, which implies, together with
\eqref{eq:z-sing-perturbed} and \th\ref{def:fully-nondegenerate},
that $\olu_\eps$ is a fully nondegenerate critical point of
$\Phi_\eps|_{\Sigma_{\abs{\olu_\eps}^2_2}}$.

To conclude the proof of
\th\ref{teo:semiclassical-existence}, it remains to compute
the Morse index of $\olu_\eps$ for $\eps >0$ small. From
\eqref{eq:z-sing-perturbed} and
\th\ref{lem:simple-but-imp}, we deduce that
\begin{equation}
  \label{eq:morse-shift-p}
  m(\olu_\eps)= m_\rmf(\olu_\eps)-1   \text{ if $2<p < 2 + \frac{4}{N}$}\quad \text{and}\quad m(\olu_\eps)=m_\rmf(\olu_\eps)   \text{ if $2 + \frac{4}{N}<p<2^*$.}          
\end{equation}
It therefore suffices to compute the free Morse index
$m_\rmf(\olu_\eps)$, which by rescaling is the same as the free
Morse index $m_\rmf(u_\eps)$ with respect to the rescaled potential
\begin{equation*}
  \Phi_\eps^* \colon H^1\to\dR,\qquad  \Phi_\eps^*(u)\coloneqq   \frac12\int_{\dR^N}(\abs{\nabla u}^2+V_\eps u^2)
  -\frac{1}{p}\int_{\dR^N}\abs{u}^p.
\end{equation*}
More precisely, the equalities in
\eqref{eq:morse-index-sing-perturbed} follow from
\eqref{eq:morse-shift-p} once we have shown that
\begin{equation}
  \label{eq:free-morse-u-eps}
  m_\rmf(u_\eps) = m_V +1  \qquad \text{for all $p \in (2,2^*)$ and $\eps>0$ small,}
\end{equation}
where $m_V$ denotes the number of negative eigenvalues of the
Hessian of $V$ at $x_0$.  The argument is partly contained in the
proof of \cite[Theorem 2.5]{MR2403325}. Nevertheless, since some
details are omitted there, we give a complete proof of
\eqref{eq:free-morse-u-eps} in
Appendix~\ref{sec:append-proof-refeq:f}. The proof of
\th\ref{teo:semiclassical-existence} is thus finished.

\section{Orbital instability}
\label{sec:orbital-instability}
This section is devoted to the proof of
\th\ref{thm:orbital-instability}.  To simplify the presentation
we only give a proof for the case $N\ge3$; the cases $N=1,2$ can
be treated similarly, slightly modifying the arguments below.

Throughout this section, we consider the special case where the
nonlinearity $f$ is odd.  We may therefore write it in the form
$f(t)= g(\abs{t}^2)t$, where $g\in C([0,\infty))\cap C^1((0,\infty))$
satisfies $g(0)=0$ and
\begin{equation*}
  \lim_{s\to\infty}\frac{g'(s)}{s^{\frac{2^*}{2}-2}}=0.
\end{equation*}
Note that in this case we have
\begin{equation*}
  \Phi(u)= \frac{1}{2} \norm{u}^2 - \int_{\dR^N}G(\abs{u}^2) =
  \frac{1}{2} \int_{\dR^N} \bigl(\abs{\nabla u}^2 +V\abs{u}^2\bigr)
  - \int_{\dR^N}G(\abs{u}^2)
\end{equation*}
for $u \in H^1$ with
$G(t)= \frac{1}{2} \int_0^{t} g$ for $t \ge 0$. To
prove the assertion on orbital instability given in
\th\ref{thm:orbital-instability}, we apply an argument
from \cite{MR1257002} with some modifications.  We identify
$\dC$ with $\dR^2$ and write the time-dependent nonlinear
Schrödinger equation \eqref{eq:time-schroedinger-0} as the
following system in $\mathbf{u}=
\begin{psmallmatrix} u_1\\u_2
\end{psmallmatrix}$ with $u_1 = {\Real}\, u$,
$u_2 = {\Imag}\, u$:
\begin{equation}
  \label{eq:time-schroedinger-system}
  \mathbf{u}_t
  = J \Bigl(- \Delta \mathbf{u} + V(x)\mathbf{u} -
  g(u_1^2+u_2^2)\mathbf{u}\Bigr)
  \qquad \text{with}\quad J\coloneqq \left(\begin{array}{cc}
      0&-1\\
      1&0  
    \end{array}
  \right).  
\end{equation}
In order to set up the functional analytic equation for this
system, we denote the dual paring between $H^{-1}$
and $H^1$  by $\scp{\cdot, \cdot}_{*}$. We put
$\cH:= H^1 \times H^1$ and write $\cH^*= H^{-1} \times H^{-1}$ for the
topological dual of $\cH$.  Recalling that we are assuming
$\min\sigma(-\Delta+V)>0$, we use the scalar product
\begin{equation*}
  \scp{u , v}_\cH = \scp{u_1,v_1} + \scp{u_2,v_2} =
  \sum_{i=1}^2 \int_{\dR^N} \bigl(\nabla u_i \cdot \nabla v_i +
  V u_i v_i\bigr),\qquad \text{for $u,v \in \cH$,}
\end{equation*}
and denote the induced norm by $\norm{\cdot}_{\cH}$. The dual
pairing between $\cH^*$ and $\cH$ is given by
\begin{equation*}
  \scp{\mathbf{u} ,\mathbf{v}}_{\cH^*,\cH} = \scp{u_1,
  v_1}_{*} + \scp{u_2, v_2}_{*} \qquad \text{for }
    \mathbf{u}=\binom{u_1}{u_2} \in \cH^*, \mathbf{v}=\binom{v_1}{v_2} \in \cH.
\end{equation*}
As usual in the context of Gelfand triples, we consider the
continuous embedding $I\colon H^1 \hookrightarrow H^{-1}$ given by
\begin{equation*}
  \scp{I u , v}_{*} : = \int_{\dR^N} u v  \qquad
  \text{for $u,v \in H^1$.}
\end{equation*}
The corresponding embedding $\cH \hookrightarrow \cH^*$ will also
be denoted by $I$, i.e., we set
\begin{equation*}
  \scp{I \mathbf{u} , \mathbf{v}}_{\cH^*,\cH} : =
  \int_{\dR^N} (u_1 v_1 + u_2 v_2) \qquad \qquad \text{for
    $\mathbf{u}=\binom{u_1}{u_2}, \mathbf{v}=\binom{v_1}{v_2} \in
    \cH$.}
\end{equation*}
With this notation, we write system
\eqref{eq:time-schroedinger-system} in the more abstract form of a
Hamiltonian system. For this we consider the functionals
\begin{equation*}
  \tilde \Phi \in C^2(\cH, \dR), \qquad \tilde \Phi(\mathbf{u})=
  \frac{1}{2}\norm{\mathbf{u}}_\cH^2- \int_{\dR^N}
  G(u_1^2+u_2^2)
\end{equation*}
and
\begin{equation*}
  \tilde \Phi_\lambda \in C^2(\cH, \dR), \qquad \tilde
  \Phi_\lambda(\mathbf{u})= \Phi(\mathbf{u})- \frac{\lambda}{2} \int_{\dR^N}
  \abs{\mathbf{u}}^2.
\end{equation*}
With this notation, \eqref{eq:time-schroedinger-system} writes as
\begin{equation*}
  (I \mathbf{u})_t = J \mathbf{\rmd } \tilde \Phi(\mathbf{u}) \qquad \text{in $\cH^*$}
\end{equation*}
where $\mathbf{\rmd } \tilde \Phi\colon \cH \to \cH^*$ denotes the
derivative of $\tilde \Phi$ and $J$ is regarded as a matrix
multiplication operator on $\cH^*= H^{-1} \times H^{-1}$.

Now let $\varphi \in \Sigma_\alpha$ satisfy the assumptions of
\th\ref{thm:orbital-instability}, and let
$\lambda \in \dR$ be the corresponding Lagrangian
multiplier. Moreover, in the following, we let
$\mathbf{\rmd ^2} \tilde \Phi_\lambda (\mathbf{\psi}) \in
\cL(\cH,\cH^*)$ denote the second derivative of
$\tilde \Phi_\lambda$ at
$\mathbf{\psi}:=
\begin{psmallmatrix} \varphi\\0
\end{psmallmatrix}
\in \cH$, which by direct
computation is given as
\begin{equation*}
  \mathbf{\rmd ^2} \tilde \Phi_\lambda  (\mathbf{\psi}) = 
  \left(\begin{array}{cc}
      L_1&0\\
      0&L_2  
    \end{array}
  \right),\qquad \text{where}\quad
  \left\{
    \begin{aligned}
      &L_1 w = -\Delta w + [V(x)-\lambda] w - f'(\varphi) w,\\
      &L_2 w = -\Delta w + [V(x)-\lambda] w - g(\abs{\varphi}^2) w.
    \end{aligned}
  \right.
\end{equation*}
Note here that $f'(t)= g(\abs{t}^2)+ 2 g'(\abs{t}^2)t^2$, so by
assumption \ref{item:3} we have $L_i \in \cL(H^1, H^{-1})$ for
$i=1,2$.  Similarly as noted in \cite[p. 187]{MR1257002}, the
orbital instability of the solitary wave solution $u_\varphi$ in
\eqref{eq:def-solitary-wave-solution-0} follows by the same
argument as in the proof of \cite[Theorem 6.2]{MR1081647} once we
have established the following.

\begin{proposition}
  \th\label{prop:orbital-instability}
  The operator
  \begin{equation*}
    \mathbf{M}:= J \mathbf{\rmd ^2} \tilde \Phi_\lambda (\mathbf{\psi})
    \in  \cL(\cH,\cH^*)
  \end{equation*} 
  has a positive real eigenvalue, i.e., there exists $\rho>0$ and
  $\mathbf{w} \in \cH \ssm \{0\}$ such that
  $\mathbf{M} \mathbf{w} = \rho I \mathbf{w}$.
\end{proposition}

The remainder of this section is devoted to the proof of
\th\ref{prop:orbital-instability}. We first note that
\begin{equation*}
  L_2 \varphi = 0 \quad \text{in $H^{-1}$,}
\end{equation*}
since $\varphi$ is a critical point of $\Phi |_{\Sigma_\alpha}$
with Lagrangian multiplier $\lambda$. Moreover, since
$\lambda < \inf \sigma_\rmess(-\Delta + V)$ by assumption, and
since $g(\abs{\varphi}^2)$ vanishes at infinity, Persson's
Theorem \cite[Theorem~14.11]{MR1361167} implies that
\begin{equation*}
  0 < \inf \sigma_\rmess(-\Delta + V-\lambda)= \inf \sigma_\rmess
  (L_2).
\end{equation*}
Since moreover $\varphi$ is a positive eigenfunction of $L_2$
corresponding to the eigenvalue $0$, it follows that
$0= \inf \sigma(L_2)$ is a simple isolated eigenvalue.
Consequently, putting
\begin{equation*}
  \tilde \Lambda:= \bigl\{v \in H^{-1} \bigm| \scp{v, \varphi
}_*=0\}  \subset   H^{-1}
\end{equation*}
and
\begin{equation*}
  \Lambda:= I^{-1}(\tilde \Lambda) = \Bigl\{v \in H^1 \Bigm|
  \int_{\dR^N} v \varphi  = 0 \Bigr\}  \subset   H^1,
\end{equation*}
we see that the quadratic form
$v \mapsto \scp{L_2 v ,v}_*$ is positive definite on
$\Lambda$ and that $L_2$ defines an isomorphism
$\Lambda \mapsto \tilde \Lambda$. From these properties, we
deduce the following.
\begin{lemma}
  \th\label{lem:orbital-instability-easy}
  We have $\scp{I L_2^{-1}I v, v}_* >0$ for all
  $v \in \Lambda \ssm \{0\}$.
\end{lemma}

\begin{proof}
  Let $v \in \Lambda \ssm \{0\}$, then
  $Iv \in \tilde \Lambda$ and by the remarks above there exists
  $\tilde v \in \Lambda \ssm \{0\}$ with
  $L_2 \tilde v = I v$. Consequently, we have
  \begin{equation*}
    \scp{I L_2^{-1}I v, v}_* = \scp{I \tilde v, v
}_* = \scp{Iv, \tilde v}_* = \scp{L_2 \tilde
    v, \tilde v}_* >0,
  \end{equation*}
  by the positive definiteness of the quadratic form
  $\tilde v \mapsto \scp{L_2 \tilde v ,\tilde v}_*$ on
  $\Lambda$.
\end{proof}

The following lemma is the key step in the proof of
\th\ref{prop:orbital-instability}. It resembles
\cite[Lemma 2.2]{MR1257002}, but we need to prove it by a
different (more general) argument since our setting does not
satisfy the assumptions in \cite{MR1257002}.

\begin{lemma}
  \th\label{lem:properties-h-lambda-2}
  We have
  \begin{equation*}
    \mu:= \inf_{v \in \Lambda \ssm \{0\}} \frac{\scp{L_1 v, v
}_*}{\scp{I L_2^{-1}Iv,v}_*} \quad \in \quad
    (-\infty,0).
  \end{equation*}
  Moreover, $\mu$ is attained at some
  $v \in \Lambda \ssm \{0\}$ satisfying the equation
  \begin{equation}
    \label{eq:v-prae-eigenvalue-eq}
    L_1 v = \mu I L_2^{-1} I v + I \beta \varphi \qquad \text{in $H^{-1}$.}
  \end{equation}
  for some $\beta \in \dR$.
\end{lemma}

\begin{proof}
  Since $\varphi$ has positive Morse index with respect to
  $\Phi |_{\Sigma_\alpha}$, there exists
  $v \in \Lambda \ssm \{0\}$ with
  $\scp{L_1 v ,v}_* < 0$, which implies that
  $\mu<0$. In the following, we consider the spectral
  decomposition
  \begin{equation*}
    \Lambda = V^- \oplus V^+
  \end{equation*}
  with the properties that $\dim V^- <\infty$ and
  \begin{equation}
    \label{eq:V+--prop}
    \scp{L_1 v, v}_* \le 0,\quad \scp{L_1 w, w}_*  \ge \delta \norm{w}^2, \quad \scp{L_1 v, w}_* = 0 \qquad \text{for $v \in V^-$, $w \in V^+$}
  \end{equation}
  with some $\delta>0$. The existence of such a decomposition
  follows from the fact that
  $\inf \sigma_\rmess (L_1)= \inf \sigma_\rmess(-\Delta +
  V-\lambda)>0.$ For $v \in \Lambda$, we now write $v= v^- + v^+$ with
  $v^- \in V^-$, $v^+ \in V^+$.  Let
  $(v_n)_n \subset \Lambda \ssm \{0\}$ be a minimizing
  sequence for the quotient
  \begin{equation*}
    v \mapsto q(v):= \frac{\scp{L_1 v, v}_*}{\scp{I
      L_2^{-1}Iv,v}_*}.
  \end{equation*}
  Since
  $\mu = \inf _{v \in \Lambda \ssm \{0\}} q(v)<0$, we
  may assume that
  \begin{equation}
    \label{eq:L1-splitting}
    \scp{L_1 v_n, v_n}_* = \scp{L_1 v_n^-, v_n^-}_* + \scp{L_1 v_n^+, v_n^+}_* < 0 \qquad \text{for all $n \in \dN$.}
  \end{equation}
  Thus $v_n^- \not = 0$, and we may assume that $\norm{v_n^-}=1$
  for all $n \in \dN$.  Since $V^-$ is finite dimensional, we may
  pass to a subsequence such that $v_n^- \to v_- \in V^-$ with
  $\norm{v_-}=1$. Then \eqref{eq:V+--prop} and
  \eqref{eq:L1-splitting} imply that
  \begin{equation*}
    \delta \limsup_{n \to \infty}\norm{v_n^+}^2 \le \limsup_{n \to
      \infty} \scp{L_1 v_n^+, v_n^+}_* \le - \lim_{n \to
      \infty} \scp{L_1 v_n^-, v_n^-}_* = -\scp{L_1 v_-,
    v_-}_*
  \end{equation*}
  and thus $v_n^+$ is bounded in $H^1$ as well. Hence
  $(v_n)_n \subset \Lambda$ is bounded in $H^1$, and we may thus
  pass to a subsequence such that
  \begin{align*}
    &v_n^+ \rightharpoonup v_+,   \qquad \qquad v_n \rightharpoonup v: = v_- + v_+\, \in \, \Lambda \ssm \{0\},\\
    &\scp{L_1 v_n, v_n}_* \to \kappa_1 \le 0 \qquad
    \text{and} \qquad \scp{I L_2^{-1} I v_n, v_n}_* \to
    \kappa_2 \ge 0
  \end{align*}
  as $n \to \infty$. By weak lower semicontinuity, we then have
  \begin{equation*}
    \scp{L_1 v_+, v_+}_* \le \lim_{n \to \infty} \scp{
    L_1 v_n^+, v_n^+}_* = \kappa_1 -\scp{L_1 v_-, v_-
}_*
  \end{equation*}
  and thus
  \begin{equation*}
    \scp{L_1 v, v}_* \le \kappa_1 \le 0.
  \end{equation*}
  Consequently, since also
  \begin{equation*}
    0< \scp{I L_2^{-1}I v, v}_* \le \kappa_2
  \end{equation*}
  by \th\ref{lem:orbital-instability-easy} and weak
  lower semicontinuity, we find that
  \begin{equation*}
    q(v) = \frac{\scp{L_1 v, v}_*}{\scp{I L_2^{-1}I v, v
}_*} \le \frac{\scp{L_1 v, v}_*}{\kappa_2} \le
    \frac{\kappa_1}{\kappa_2} = \mu.
  \end{equation*}
  Hence $v$ is a minimizer of $q$ in $\Lambda \ssm \{0\}$,
  and therefore $q(v)=\mu> -\infty$.  Moreover, $v$ minimizes the
  functional
  \begin{equation*}
    \Lambda \to \dR, \qquad w \mapsto \scp{L_1 w-\mu I L_2^{-1} I
    w,w}_*
  \end{equation*}
  and therefore we have
  \begin{equation*}
    \scp{L_1 v - \mu I L_2^{-1} Iv , w}_* = 0 \qquad
    \text{for all $w \in \Lambda$.}
  \end{equation*}
  This implies that there exists $\beta \in \dR$ such that
  \begin{equation*}
    \scp{L_1 v - \mu I L_2^{-1} Iv , w}_* = \beta
    \int_{\dR} \varphi w \qquad \text{for all $w \in H^1$,}
  \end{equation*}
  i.e.,
  \begin{equation*}
    L_1 v - \mu I L_2^{-1} I v = \beta I \varphi \qquad \text{in
      $H^{-1}$,}
  \end{equation*}
  which gives \eqref{eq:v-prae-eigenvalue-eq}.
\end{proof}

\begin{proof}[Proof of \th\ref{prop:orbital-instability}
  (completed)]
  Let $\mu$ and $v$ be as in
  \th\ref{lem:properties-h-lambda-2}, let
  $\rho = \sqrt{-\mu}>0$, and consider
  \begin{equation*}
    \mathbf{w}= \binom{v}{-\rho L_2^{-1} I v + \rho^{-1}\beta
      \varphi} \in \cH \ssm \{0\}.
  \end{equation*}
  Then we have
  \begin{equation*}
    \mathbf{M} \mathbf{w}= \left(\begin{array}{cc}
        0&-L_2\\
        L_1&0
      \end{array}
    \right) \mathbf{w} = \binom{\rho I v}{\mu I L_2^{-1} I v + I
      \beta \varphi} = \rho I \mathbf{w},
  \end{equation*}  
  so $\mathbf{w} \in \cH$ is an eigenfunction of $\mathbf{M}$
  corresponding to the eigenvalue $\rho>0$.
\end{proof}

\appendix
\section{Proof of \eqref{eq:free-morse-u-eps}}
\label{sec:append-proof-refeq:f}
In this section we compute the free Morse index of the rescaled
single peak solutions $u_\eps$ of \eqref{eq:66-1} studied in
Section~\ref{sec:example}. More precisely, we will prove the
equality \eqref{eq:free-morse-u-eps} for $\eps>0$ small. We
continue to use the notation from
Section~\ref{sec:example}. Recall that since $u_\varepsilon$ is a
critical point of $\Phi_\varepsilon^*$ on
$\sum_{\abs{u_\varepsilon}_2^2}$ with Lagrange multiplier $0$,
the free Morse index coincides with the Morse index of
$u_\varepsilon$ as a critical point of $\Phi_\varepsilon^*$ in
$H^1$.  Recall moreover that $u_\eps$ has a unique local maximum
point $x_\eps$, where $x_\eps \to 0$ as $\eps \to 0$ by
\cite[Proposition 5.2]{MR1956951}. Put
\begin{equation*}
  u_{0,\eps}:= u_0( \cdot - x_\eps) = \cT_{x_\eps} u_0 \in H^1
  \qquad \text{for $\eps \in (0,\eps_0)$.}
\end{equation*}
We first need the following refined convergence estimate:
\begin{equation}
  \label{eq:H2-conv-unif-refined}
  \norm{u_{0,\eps} - u_\eps}_{H^2} = O(\eps^2) \qquad \text{as $\eps \to 0$.}
\end{equation}
Suppose by contradiction that this is false, then along a
sequence $(\eps_n)_n \subset (0,\eps_0)$ with $\eps_n \to 0$ we
have
$d_n:= \norm{u_{0,\eps_n} - u_{\eps_n}}_{H^2} \ge n \eps_n^2$ for
all $n \in \dN$. Put $w_n:= \frac{u_{0,\eps_n}-u_{\eps_n}}{d_n}$;
then $w_n$ is a weak solution of the equation
\begin{equation}
  \label{eq:w_n}
  -\Delta w_n + w_n = \frac{1}{d_n} \Bigl( u_{0,\eps_n}^{p-1} -u_{\eps_n}^{p-1} + (V_{\eps_n}-1)u_{\eps_n}\Bigr)= \tau_n w_n + \frac{V_{\eps_n}-1}{d_n}u_{\eps_n}
\end{equation}
with
\begin{equation*}
  \tau_n(x)= (p-1) \int_0^1 [(1-s)u_{0,\eps_n} +
  su_{\eps_n}]^{p-2}\dint s.
\end{equation*}
We pass to a subsequence such that $w_n \weakto w$ in
$H^2$. Since $\tau_n \to (p-1) u_0^{p-2}$ as $n \to \infty$
uniformly in $\dR^N$ by \eqref{eq:H2-conv-unif}, and since
\begin{equation}
  \label{eq:est-v-d-n}
  \xabs{ \frac{V_{\eps_n}-1}{d_n}u_{\eps_n}(x)} \le \frac{c}{n} \abs{x}^2 \rme^{-\alpha \abs{x}} \qquad \text{for $x \in \dR^N$, $n \in \dN$ with constants $c,\alpha>0$}
\end{equation}
by \eqref{eq:uniform-exp} and \eqref{eq:Veps-limit-1}, we may
pass to the limit in \eqref{eq:w_n} to see that $w$ is a (weak)
solution of the equation
\begin{equation*}
  -\Delta w + w - (p-1)u_0^{p-2} w = 0.
\end{equation*}
Consequently,
$w = \sum _{\ell=1}^N a_\ell \partial_\ell u_0 $ with
$\ell =1,\dots,N$. However, since both $u_{0,\eps_n}$ and
$u_{\eps_n}$ attain a maximum at $x_{\eps_n}$, we infer from
\eqref{eq:w_n} and elliptic regularity that
\begin{equation*}
  0= \lim_{n \to \infty} \partial_j w_n(x_{\eps_n})= \partial_j
  w(0)= \sum_{\ell=1}^N a_\ell \partial_{\ell j} u_0(0) \qquad
  \text{for $j=1,\dots,N$.}
\end{equation*}
It is well known that $0$ is the only maximum point of $u_0$,
see, e.g., \cite[Lemma~1(b)]{MR94b:35105}.
Considering that $u_0(x)=U_0(\abs{x})$, where $U_0$ is the
solution with initial values $U_0(0)=u_0(0)$ and $U_0'(0)=0$ of
the ordinary differential equation on $[0,\infty)$ corresponding
to radial solutions of \eqref{eq:limit-equation}, and considering
the uniqueness of solutions to that ODE, it is clear that $0$ is
a nondegenerate maximum point for $u_0$.  Hence it follows that
$a_1,\dots,a_N=0$ and thus $w=0$. This implies that $w_n \to 0$
in $L^2_{\rmloc}(\dR^N)$, and thus
\begin{equation*}
  -\Delta w_n + w_n = o(1) \qquad \text{in $L^2(\dR^N)$}
\end{equation*}
by \eqref{eq:w_n}, \eqref{eq:est-v-d-n}, and since $\tau_n$ has
exponential decay in $x$, uniformly in $n$.  The boundedness of the
inverse of $-\Delta+1$ on $L^2$ implies that
$\norm{w_n}_{H^2}\to0$, contrary to the definition of
$w_n$. Hence \eqref{eq:H2-conv-unif-refined} follows.

We now consider the uniformly bounded families of linear
operators
\begin{equation*}
  A_\eps := \rmD^2 \Phi_\eps^*(u_\eps)   \in  \cL(H^1)
\end{equation*}
and
\begin{equation*}
  C_\eps := \cT_{-x_\eps} \circ A_\eps \circ \cT_{x_\eps}  \in  
  \cL(H^1), \qquad \eps \in (0,\eps_0).
\end{equation*}
Here, as before, the symbol $\rmD^2$ denotes the derivative of the gradient
with respect to the scalar product $\scp{\cdot,\cdot}_{H^1}$.  The
quadratic form associated with $A_\eps$ is given by
\begin{equation}
  \label{eq:quadratic-A-eps}
  \scp{A_\eps v, w}_{H^1} = \int_{\dR^N} \bigl( \nabla v \cdot \nabla w
  + [V_\eps - (p-1)u_\eps^{p-2}] v w\bigr) \qquad \text{for $v,w \in H^1$.}
\end{equation}
It is then clear that $A_\eps$ and $C_\eps$ share the same
spectrum.  We have
\begin{equation}
  \label{eq:strong-convergence-c-eps}
  \lim_{\eps \to 0} \norm{C_\eps v  - B^0 v}_{H^1} =\lim_{\eps \to 0} \norm{A_\eps v  - B^0 v}_{H^1} = 0  \qquad \text{for all $v \in H^1$,}
\end{equation}
where, as before, $B^0= \rmD^2 \Phi_0^*(u_0) \in \cL(H^1)$, and
the convergence is uniform on compact subsets of $H^1$. We claim
that
\begin{equation}
  \label{eq:C-eps-partial-der-0}
  \norm{C_\eps \partial_i u_0}_{H^1} = O(\eps^2) \qquad \text{for $i=1,\dots,N$,}
\end{equation}
and that
\begin{equation}
  \label{eq:C-eps-partial-der}
  \scp{C_\eps \partial_i u_0, \partial_j u_0}_{H^1}
  =\frac12 \eps^2  \partial_{ij}  V(0)  \abs{u_0}_2^2 +o(\eps^2) \qquad \text{for $i,j=1,\dots,N$} 
\end{equation}
as $\eps \to 0$. For this we recall that $\partial_i u_\eps$
solves the equation
\begin{equation}
  \label{eq:partial-j-equation}
  -\Delta (\partial_i u_{\eps}) + V_\eps \partial_j u_\eps  - (p-1)u_\eps^{p-2}\partial_j u_\eps= -  u_\eps \partial_j V_{\eps},
\end{equation}
and therefore \eqref{eq:uniform-exp} and \eqref{eq:Veps-limit-2}
yield
\begin{equation}
  \begin{aligned}
    A_\eps \partial_i u_{\eps} = (-\Delta + 1)^{-1}\Bigl( -\Delta (\partial_i u_{\eps}) &+ V_\eps \partial_i u_\eps  - (p-1)u_\eps^{p-2}\partial_i u_\eps\Bigr) \\
    &= - (-\Delta + 1)^{-1} u_\eps \partial_j V_{\eps} =
    O(\eps^2) \qquad \text{in $H^1$.}
  \end{aligned}
  \label{eq:A-eps}
\end{equation}
Combining this with \eqref{eq:H2-conv-unif-refined}, we find that
\begin{equation*}
  \norm{C_\eps \partial_i u_0}_{H^1} = \norm{A_\eps \partial_i u_{0,\eps}}_{H^1} =
  \norm{A_\eps \partial_i u_{\eps}}_{H^1} + O(\eps^2) = O(\eps^2),
\end{equation*}
as claimed in \eqref{eq:C-eps-partial-der-0}. To see
\eqref{eq:C-eps-partial-der}, we note that
\begin{equation}
  \begin{aligned}
    \scp{& C_\eps \partial_i u_0, \partial_j u_0}_{H^1}
    = \scp{A_\eps \partial_i u_{0,\eps}, \partial_j u_{0,\eps}}_{H^1} \\
    &= \scp{A_\eps \partial_i u_{\eps}, \partial_j u_{\eps}}_{H^1}
    + \scp{A_\eps \partial_i u_{0,\eps}, \partial_j
      (u_{0,\eps}-u_{\eps})}_{H^1} + \scp{A_\eps \partial_j
      u_{\eps}, \partial_i (u_{0,\eps}-u_{\eps})}_{H^1},
  \end{aligned}
\label{eq:C-eps-partial-der-1}
\end{equation}
where, since $\partial_ i u_{0,\eps}$ satisfies
$-\Delta \partial_ i u_{0,\eps} + \partial_ i u_{0,\eps} -
(p-1)u_{0,\eps}^{p-2} \, \partial_ i u_{0,\eps} =0$ in $\dR^N$,
\begin{equation*}
  \scp{A_\eps \partial_i u_{0,\eps}, \partial_j
  (u_{0,\eps}-u_{\eps})}_{H^1} = \int_{\dR^N} \bigl [V_\eps-1+
  (p-1) (u_{0,\eps}^{p-2}-u_\eps^{p-2})\bigr] \partial_i
  u_{0,\eps}\,
  \partial_j (u_{0,\eps}-u_{\eps})= o(\eps^2)
\end{equation*}
as $\eps \to 0$. Here, in the last step, we used
\eqref{eq:H2-conv-unif-refined} together with the fact that
\begin{equation*}
  \norm{[V_\eps-1+ (p-1)
    (u_{0,\eps}^{p-2}-u_\eps^{p-2})\bigr] \partial_i
    u_{0,\eps}}_{L^2} \to 0 \qquad \text{as $\eps \to 0$.}
\end{equation*}
Moreover,
\begin{equation*}
  \abs{\scp{A_\eps \partial_j u_{\eps}, \partial_i
    (u_{0,\eps}-u_{\eps})}_{H^1}} \le \norm{ A_\eps \partial_j
    u_{\eps}}_{H^1} \norm{\partial_i (u_{0,\eps}-u_{\eps})}_{H^1} \le O(\eps^4)
\end{equation*}
by \eqref{eq:H2-conv-unif-refined} and \eqref{eq:A-eps}.
Inserting these estimates in \eqref{eq:C-eps-partial-der-1} and
using \eqref{eq:partial-j-equation} once more, together with
\eqref{eq:uniform-exp}, \eqref{eq:H2-conv-unif}, and
\eqref{eq:Veps-limit-2} we find that
\begin{align*}
  \scp{C_\eps \partial_i u_0, \partial_j u_0}_{H^1}&= \scp{
  A_\eps \partial_i u_{\eps}, \partial_j u_{\eps}}_{H^1}+
  o(\eps^2)
  = - \int_{\dR^N} u_{\eps} \partial_i V_{\eps}   \, \partial_j u_{\eps}+o(\eps^2)\\
  &= - \eps^2 \sum_{\ell = 1}^N \partial_{i\ell} V(0)
  \int_{\dR^N} x_\ell u_0 \, \partial_j u_0\dint x +o(\eps^2)
  = \frac12\eps^2 \partial_{ij} V(0) \abs{u_0}_2^2 +o(\eps^2).
\end{align*}
In the last step we have integrated by parts again. This yields
\eqref{eq:C-eps-partial-der}.

To conclude the proof of \eqref{eq:free-morse-u-eps}, we now put
$X= \opspan(u_0)$, $Y:= \opspan(\partial_1 u_0,\dots, \partial_N u_0)$, and we
let $Z$ denote the $\scp{\cdot,\cdot}_{H^1}$-orthogonal complement of
$X \oplus Y$ in $H^1$. We then have the
$\scp{\cdot,\cdot}_{H^1}$-orthogonal decomposition
$H^1= X \oplus Y \oplus Z$, and we let $P_X,P_Y,P_Z \in \cL(H^1)$
denote the corresponding orthogonal projections onto $X$, $Y$,
and $Z$.  It then follows from \eqref{eq:C-eps-partial-der-0}
that
\begin{equation}
  \label{eq:P-Y-estimate}
  \norm{C_\eps P_Y}_{\cL(H^1)} = O(\eps^2) \qquad \text{as $\eps \to 0$.}
\end{equation}
Moreover, by the remarks before \th\ref{lem:criterion}, there
exists $0<\delta<1$ such that
\begin{equation}
  \label{eq:unperturbed-est-spectral}
  \scp{B^0 u_0, u_0}_{H^1} \le - \delta \qquad \text{and}\qquad \scp{B^0 w, w}_{H^1} \ge \delta \norm{w}_{H^1}^2\quad \text{for all $w \in Z$.}
\end{equation}
It then follows from \eqref{eq:strong-convergence-c-eps} that
\begin{equation}
  \label{eq:lower-min-max-0}
  \scp{C_\eps u_0, u_0}_{H^1} < - \frac{\delta}{2} \qquad \text{for $\eps>0$ sufficiently small.}
\end{equation}
We also claim that
\begin{equation}
  \label{eq:upper-min-max-0}
  \inf_{w \in Z,\norm{w}_{H^1}=1} \scp{C_\eps w, w}_{H^1} > \delta_+ := \frac{1}{2}\min \{\delta, \inf_{\dR^N} V\} \qquad \text{for $\eps>0$ sufficiently small.}
\end{equation}
Indeed, suppose by contradiction there exist
$\eps_n \in (0,\eps_0)$ and $w_n \in Z$ with $\norm{w_n}_{H^1}=1$ for
$n \in \dN$ such that $\eps_n \to 0$ as $n \to \infty$ and
\begin{equation}
  \label{eq:upper-min-max-0-contra}
  \scp{C_{\eps_n} w_n, w_n}_{H^1} \le \delta_+ \qquad \text{as $n \to \infty$.}
\end{equation}
Passing to a subsequence, we may then assume that $w_n \weakto w$
in $H^1$ with $w \in Z$. We put
$\tilde w_n := \cT_{x_{\eps_n}} w_n = w_n (\cdot - x_{\eps_n})$
for $n \in \dN$, then also $\tilde w_n \weakto w$, and we may
pass to a subsequence such that $\tilde w_n \to w$ in
$L^2_{\loc}(\dR^N)$ and $\tilde w_n \to w$ pointwise a.e. on
$\dR^N$. By \eqref{eq:uniform-exp} and \eqref{eq:H2-conv-unif}
this implies that
\begin{equation}
  \label{eq:loc-conv-w}
  \int_{\dR^N}u_{\eps_n}^{p-2} \tilde w_n^2  \to \int_{\dR^N}u_{0}^{p-2} w^2  \qquad \text{as $n \to \infty$.} 
\end{equation}
We also have that
\begin{equation*}
  \int_{\dR^N} \!\bigl(\abs{\nabla (\tilde w_n-w)}^2 +
  V_{\eps_n}(\tilde w_n-w)^2\bigr) = \!o(1)+ \int_{\dR^N}\!
  \bigl(\abs{\nabla \tilde w_n}^2 - \abs{\nabla w}^2 + V_{\eps_n}[\tilde
  w_n^2- w^2 - 2(\tilde w_n-w)w] \bigr),
\end{equation*}
where, since $\abs{\tilde w_n-w} \weakto 0$ in $L^2(\dR^N)$,
\begin{equation*}
  \Bigabs{\int_{\dR^N} V_{\eps_n} (\tilde w_n-w) w } \le
  \norm{V}_{L^\infty(\dR^N)} \int_{\dR^N} \abs{\tilde w_n-w} \abs{w}  \to
  0 \qquad \text{as $n \to \infty$.}
\end{equation*}
Moreover,
\begin{equation*}
  \int_{\dR^N} V_{\eps_n}w^2 \to \int_{\dR^N} w^2 \qquad
  \text{as $n \to \infty$}
\end{equation*}
by \eqref{eq:Veps-limit-1} and Lebesgue's theorem. Consequently,
\begin{align*}
  \int_{\dR^N} &\bigl(\abs{\nabla \tilde w_n}^2 + V_{\eps_n}\tilde w_n^2\bigr)\\
  &= \int_{\dR^N} (\abs{\nabla w}^2 + w^2) + \int_{\dR^N} \bigl(\abs{\nabla (\tilde w_n-w)}^2 + V_{\eps_n} (\tilde w_n-w)^2\bigr)+o(1)\\
  &\ge \norm{w}_{H^1}^2 + \min \{1, \inf_{\dR^N} V\} \norm{\tilde
    w_n-w}_{H^1}^2+o(1)
  \ge\norm{w}_{H^1}^2 + 2\delta_+ \norm{\tilde
    w_n-w}_{H^1}^2+o(1) ,
\end{align*}
and together with \eqref{eq:quadratic-A-eps},
\eqref{eq:unperturbed-est-spectral} and \eqref{eq:loc-conv-w}
this implies that
\begin{align*}
  \scp{&C_{\eps_n} w_n, w_n}_{H^1}
  = \scp{A_{\eps_n}  \tilde w_n, \tilde w_n}_{H^1} \ge \scp{B^0 w,w}_{H^1} +
  2\delta_+  \norm{\tilde w_n-w}_{H^1}^2+o(1)\\
  & \ge 2\delta_+ \norm{w}_{H^1}^2 + 2 \delta_+ \norm{\tilde w_n-w}_{H^1}^2+o(1)
  = 2 \delta_+\norm{w_n}_{H^1}^2+ o(1)
  = 2 \delta_+ +o(1).
\end{align*}
This contradicts \eqref{eq:upper-min-max-0-contra}, and hence
\eqref{eq:upper-min-max-0} follows.

In the following, we let $M \in \dR^{N \times N}$ denote the
Hessian of the potential $V$ at $0$ which is nondegenerate by
assumption. Then there exists a basis of eigenvectors
$b^1,\dots,b^N \in \dR^N$ of $M$ corresponding to the eigenvalues
$\mu_1 \le \dots \le \mu_N$, where
\begin{equation*}
  \mu_{i}<0 \quad \text{for $i \le m_V$}\quad \text{and}\quad
  \mu_{i}>0 \quad \text{for $i > m_V$.}
\end{equation*}
We then let
$w^1,\dots,w^N \in \opspan(\partial_1 u_0,\dots, \partial_N u_0)$ be
defined by
\begin{equation*}
  w^i:= \sum_{j=1}^N b^i_j \partial_j u_0 \qquad \text{for
    $i=1,\dots,N$,}
\end{equation*}
and we define the subspaces $\tilde Y_\pm \subset Y$ by
\begin{equation*}
  \tilde Y_-:=\opspan(w^1,\dots,w^{m}) \qquad \text{and}\qquad \tilde
  Y_+:=\opspan(w^{m+1},\dots,w^N).
\end{equation*}
By \eqref{eq:C-eps-partial-der} and construction, there exists
$\tilde \delta>0$ such that for $\eps>0$ sufficiently small we
have
\begin{equation}
  \label{eq:tilde-delta}
  \scp{C_{\eps} w, w}_{H^1} \le - \tilde \delta \eps^2 \norm{w}_{H^1}^2  \text{ for $w \in \tilde Y_-$}\quad \text{and}\quad \scp{C_{\eps} w, w}_{H^1} \ge  \tilde \delta \eps^2 \norm{w}_{H^1}^2  \text{ for $w \in \tilde Y_+$.}  
\end{equation}
We now consider the spaces
\begin{equation*}
  \tilde X:= \opspan(u_0) \oplus \tilde Y_- \qquad \text{and}\qquad \tilde
  Z:= Z \oplus \tilde Y_+.
\end{equation*}
Then \eqref{eq:free-morse-u-eps} follows once we have shown that
\begin{equation}
  \label{eq:lower-min-max}
  \sup_{w \in \tilde X,\norm{w}_{H^1}=1} \scp{C_\eps w, w}_{H^1} <0
\end{equation}
and
\begin{equation}
  \label{eq:upper-min-max}
  \inf_{w \in \tilde Z,\norm{w}_{H^1}=1} \scp{C_\eps w, w}_{H^1} >0
\end{equation}
for $\eps> 0$ sufficiently small.  We only show
\eqref{eq:upper-min-max}, the proof of \eqref{eq:lower-min-max}
is very similar but simpler. Suppose by contradiction that
\eqref{eq:upper-min-max} does not hold true for $\eps> 0$ sufficiently
small. Then there exist $\eps_n \in (0,\eps_0)$ and
$w_n \in \tilde Z$ with $\norm{w_n}_{H^1}=1$ for $n \in \dN$ such that
$\eps_n \to 0$ as $n \to \infty$ and
\begin{equation}
  \label{eq:upper-min-max-contra}
  \scp{C_{\eps_n} w_n, w_n}_{H^1} \le 0 \qquad \text{as $n \to \infty$.}
\end{equation}
With $w_n^1:= P_Z w_n \in Z$ and
$w_n^2 := P_Y w_n \in \tilde Y^+$ we have, by
\eqref{eq:P-Y-estimate}, \eqref{eq:upper-min-max-0} and
\eqref{eq:tilde-delta},
\begin{align*}
  \scp{C_{\eps_n} w_n, w_n}_{H^1}&= \scp{C_{\eps_n} w_n^1
  ,w_n^1}_{H^1} + \scp{C_{\eps_n} w_n^2 , w_n^2}_{H^1} +
  2  \scp{C_{\eps_n} w_n^2  , w_n^1}_{H^1}\\
  &\ge \delta_+ \norm{w_n^1}_{H^1}^2 + \tilde \delta \norm{w_n^2 }_{H^1}^2
  \eps_n^2 + O(\norm{w_n^1}_{H^1} \eps_n^2).
\end{align*}
Passing to a subsequence, we may assume that either
$\norm{w_n^1}_{H^1} \to 0$ and $\norm{w_n^2}_{H^1} \to 1$ as $n \to \infty$,
or that $\norm{w_n^1}_{H^1} \ge c$ for some constant $c>0$ and all
$n \in \dN$. In the first case, we deduce that
\begin{equation*}
  \scp{C_{\eps_n} w_n, w_n}_{H^1} \ge \tilde \delta \eps_n^2 +
  o(\eps_n^2)
\end{equation*}
and in the second case we obtain that
\begin{equation*}
  \scp{C_{\eps_n} w_n, w_n}_{H^1} \ge \delta_+ c^2 + o(1)
\end{equation*}
as $n \to \infty$.  In both cases we arrive at a contradiction
to \eqref{eq:upper-min-max-contra}, and thus
\eqref{eq:upper-min-max} is proved. As remarked before,
\eqref{eq:lower-min-max} is obtained similarly by using
\eqref{eq:lower-min-max-0} and the first inequality in
\eqref{eq:tilde-delta}. The proof of \eqref{eq:free-morse-u-eps}
is thus finished.

\def\cprime{$'$} \def\polhk#1{\setbox0=\hbox{#1}{\ooalign{\hidewidth
  \lower1.5ex\hbox{`}\hidewidth\crcr\unhbox0}}}
\providecommand{\bysame}{\leavevmode\hbox to3em{\hrulefill}\thinspace}
\providecommand{\MR}{\relax\ifhmode\unskip\space\fi MR }
\providecommand{\MRhref}[2]{%
  \href{http://www.ams.org/mathscinet-getitem?mr=#1}{#2}
}
\providecommand{\href}[2]{#2}

\end{document}